\documentclass[11pt]{article}

\usepackage{graphicx}
\usepackage{mathtools,tikz}
\usepackage{amsthm,amssymb,amsmath,mathtools}
\usepackage{float}
\usepackage{hyperref}
\usepackage{cleveref}
\usetikzlibrary{calc}
\usetikzlibrary{fit}
\usetikzlibrary{decorations}
\usetikzlibrary{decorations.pathmorphing}
\usetikzlibrary{decorations.text}
\usepackage{graphicx}
\usepackage[shortlabels]{enumitem}
\usepackage{mathabx}
\usepackage{bm}
\usetikzlibrary{shapes,hobby}

\usepackage{hyperref}
\usepackage{cleveref}

\usepackage{thmtools}
\usepackage{thm-restate}

\usepackage[margin=2.5cm]{geometry}

\colorlet{myGreen}{green!50!black}
\colorlet{myLightgreen}{green}
\colorlet{myRed}{red!90!black}
\definecolor{myBlue}{rgb}{0.25, 0.0, 1.0}
\definecolor{myLightBlue}{rgb}{0.39, 0.58, 0.93}
\colorlet{myViolet}{myBlue!55!myRed}
\definecolor{myOrange}{rgb}{1.0, 0.66, 0.07}

\definecolor{CornflowerBlue}{rgb}{0.39, 0.58, 0.93}
\definecolor{DarkGoldenrod}{rgb}{0.72, 0.53, 0.04}
\definecolor{BritishRacingGreen}{rgb}{0.0, 0.26, 0.15}
\definecolor{DarkMagenta}{rgb}{0.55, 0.0, 0.55}
\definecolor{AO}{rgb}{0.0, 0.5, 0.0}
\definecolor{BostonUniversityRed}{rgb}{0.8, 0.0, 0.0}
\definecolor{myRed}{rgb}{0.8, 0.0, 0.0}
\definecolor{DarkMidnightBlue}{rgb}{0.0, 0.2, 0.4}
\definecolor{DarkTangerine}{rgb}{1.0, 0.66, 0.07}
\definecolor{AppleGreen}{rgb}{0.55, 0.71, 0.0}
\definecolor{BrightUbe}{rgb}{0.82, 0.62, 0.91}
\definecolor{Amethyst}{rgb}{0.6, 0.4, 0.8}
\definecolor{DarkGray}{rgb}{0.52, 0.52, 0.51}
\definecolor{Gray}{rgb}{0.66, 0.66, 0.66}
\definecolor{BananaYellow}{rgb}{1.0, 0.88, 0.21}
\definecolor{Amber}{rgb}{1.0, 0.75, 0.0}
\definecolor{LightGray}{rgb}{0.83, 0.83, 0.83}
\definecolor{PrincetonOrange}{rgb}{1.0, 0.56, 0.0}
\definecolor{DeepCarrotOrange}{rgb}{0.91, 0.41, 0.17}
\definecolor{CarrotOrange}{rgb}{0.93, 0.57, 0.13}
\definecolor{MidnightBlue}{rgb}{0.1, 0.1, 0.44}
\definecolor{Magenta}{rgb}{0.50, 0.0, 0.50}
\definecolor{BrightPink}{rgb}{1.0, 0.0, 0.5}
\definecolor{BrilliantRose}{rgb}{1.0, 0.33, 0.64}
\definecolor{ChromeYellow}{rgb}{1.0, 0.65, 0.0}
\definecolor{HotMagenta}{rgb}{1.0, 0.11, 0.81}
\definecolor{Amethyst}{rgb}{0.6, 0.4, 0.8}

\hypersetup{
colorlinks=true,
linkcolor=AO!65!black,
citecolor=AO!65!black,
urlcolor=AppleGreen!65!black,
bookmarksopen=true,
bookmarksnumbered,
bookmarksopenlevel=2,
bookmarksdepth=3
}

\newtheorem{environment}{Environment}[section]

\newtheorem{lemma}[environment]{Lemma}
\crefname{lemma}{lemma}{lemmata}
\crefformat{lemma}{#2Lemma~#1#3}
\Crefformat{lemma}{#2Lemma~#1#3}

\newtheorem{observation}[environment]{Observation}
\crefname{observation}{observation}{observations}
\crefformat{observation}{#2Observation~#1#3}
\Crefformat{observation}{#2Observation~#1#3}

\newtheorem{theorem}[environment]{Theorem}
\crefname{theorem}{theorem}{Theorems}
\crefformat{theorem}{#2Theorem~#1#3}
\Crefformat{theorem}{#2Theorem~#1#3}

\crefname{section}{section}{sections}
\crefformat{section}{#2Section~#1#3}
\Crefformat{section}{#2Section~#1#3}

\newtheorem{definition}[environment]{Definition}
\crefname{definition}{definition}{definitions}
\crefformat{definition}{#2Definition~#1#3}
\Crefformat{definition}{#2Definition~#1#3}

\newtheorem{claim}{Claim}[environment]
\crefname{claim}{claim}{claims}
\crefformat{claim}{#2Claim~#1#3}
\Crefformat{claim}{#2Claim~#1#3}

\newtheorem{conjecture}[environment]{Conjecture}
\crefname{conjecture}{conjecture}{Conjectures}
\crefformat{conjecture}{#2Conjecture~#1#3}
\Crefformat{conjecture}{#2Conjecture~#1#3}

\newtheorem{corollary}[environment]{Corollary}
\crefname{corollary}{corollary}{corollaries}
\crefformat{corollary}{#2Corollary~#1#3}
\Crefformat{corollary}{#2Corollary~#1#3}

\crefname{figure}{figure}{figures}
\crefformat{figure}{#2Figure~#1#3}
\Crefformat{figure}{#2Figure~#1#3}

\DeclareMathOperator{\atw}{\alpha\mbox{-}\mathsf{tw}}

\title{Excluding a Ladder as an Induced Minor in Graphs Without Induced Stars}

\author{\medskip Mujin Choi\thanks{Supported by the Institute for Basic Science (IBS-R029-C1)}$\hspace{0.19cm} ^{1,}$$^{2}$ and Sebastian Wiederrecht\thanks{Sebastian Wiederrecht's research was partially supported by the Institute for Basic Science (IBS-R029-C1).}$\hspace{0.18cm}^{3}$\\
        \small $^{1}$Department of Mathematical Sciences, KAIST, Daejeon, Korea \\
        \small $^{2}$Discrete Mathematics Group, Institute for Basic Science (IBS), Daejeon, South Korea \\
        \small $^{3}$School of Computing, KAIST, Daejeon, Korea \\
        \small Email: \href{mailto:mujinchoi@kaist.ac.kr}{mujinchoi@kaist.ac.kr}, \href{mailto:wiederrecht@kaist.ac.kr}{wiederrecht@kaist.ac.kr} \\
}
\date{\today}

\begin{document}

\maketitle

\begin{abstract}

A $k$-ladder is the graph obtained from two disjoint paths, each with $k$ vertices, by joining the $i$th vertices of both paths with an edge for each $i\in\{ 1,\ldots,k\}$.
In this paper, we show that for all positive integers $k$ and $d$, the class of all $K_{1,d}$-free graphs excluding the $k$-ladder as an induced minor has a bounded tree-independence number.

We further show that our method implies a number of known results:
We improve the bound on the tree-independence number for the class of $K_{1,d}$-free graphs not containing a wheel as an induced minor given by Choi, Hilaire, Milanič, and Wiederrecht \cite{choi2025excludinginducedwheelminor}.
Furthermore, we show that the class of $K_{1,d}$-free graphs not containing a theta or a prism, whose paths have length at least $k$, as an induced subgraph has bounded tree-independence number.
This improves a result by Chudnovsky, Hajebi, and Trotignon \cite{chudnovsky2024treeindependencenumberiii}.
Finally, we extend the induced Erd\H{o}s-P{\'o}sa result of Ahn, Gollin, Huynh, and Kwon in $K_{1,d}$-free graphs from long induced cycles to any graph that is an induced minor of the $k$-ladder where every edge is subdivided exactly once.
\end{abstract}

\section{Introduction}\label{sec:Intro}
The Grid Theorem by Robertson and Seymour is a fundamental theorem in modern structural graph theory.
The theorem states that there exists a function $f\colon\mathbb{N}\to\mathbb{N}$ such that for every positive integer $k$, every graph with treewidth at least $f(k)$ contains the $(k\times k)$-grid as a minor.
The theorem has a wide range of applications in both structural algorithmic graph theory \cite{Thilikos2015Bidimensionality,Korhonen2024Minorcontainmentinalmostlineartime}.

Due to its pivotal role in the (algorithmic) theory of graph minors, there have been numerous attempts (successful and unsuccessful) to find similar theorems for other minor-like relations combined with an appropriate width parameter.
In \cite{Directedgridtheorem}, Kawarabayashi and Kreutzer proved an analogue of the Grid Theorem for directed graphs where the role of treewidth is taken on by directed treewidth and the containment relation is called butterfly minors.
In \cite{GEELENgridtheoremforvertexminor2023}, Geelen, Kwon, McCarty, and Wollan proved that any grpah of large rankwidth contains a graph called the ``combarability grid'' as a vertex minor.
In \cite{Wollan2015Structure}, Wollan proved that graph of sufficiently large tree-cut width contain a large wall -- a subcubic variant of the grid -- as an immersion.
This theorem was recently strengthened to also hold for strong immersions by Diestel, Jacobs, Knappe, and Wollan \cite{diestel2025gridtheoremstrongimmersions} with a slight tweak to the definition of tree-cut-width.

All of the containment relations above have in common that they allow for the deletion of edges.
Finding appropriate generalizations of the Grid Theorem for the setting of induced subgraphs\footnote{For graphs $G$ and $H$, we say that $H$ is an \emph{induced subgraph} of $G$ if $H$ can be obtained from $G$ by a sequence of vertex deletions.
We say that $G$ is \emph{$H$-free} if $G$ has no induced subgraph isomorphic to $H$.} or \textsl{induced minors} is an ongoing and seemingly much harder challenge.

For graphs $G$ and $H$, we say that $H$ is an \emph{induced minor} of $G$ if $H$ can be obtained from $G$ by a sequence of vertex deletions and edge contractions.
Indeed, for the setting the induced minors, the picture is somewhat hazy.
It appears to be unlikely that there is a unique parametric graph -- such as the grid -- which obstructs large treewidth when considering induced minors (or induced subgraphs), see \cite{alecu2025graphessentiallargetreewidth} for example.
On the other hand, there are multiple good -- and pairwise distinct -- options for defining a variant of treewidth fit for the study of the induced setting  \cite{dallard2024treewidthversuscliquenumber2, nguyen2025coarsetreewidth}.
These alternatives to treewidth appear to exhibit a slightly tamer behavior under certain circumstances.
For this paper we focus on the notion called \textsl{tree-independence number} or \textsl{$\alpha$-treewidth} \cite{dallard2024treewidthversuscliquenumber2,YolovMinorMatchingHypertreeWidth}.

Recently, the problem of understanding hereditary graph classes of large $\alpha$-treewidth has drawn much attention, and several partial results have emerged  \cite{treeindependencenumberievenholediamondpyramid,dallard2024treewidthversuscliquenumber4,chudnovsky2024treeindependencenumberiii}.
In the setting of \textsl{bounded degree}, Korhonen \cite{Korhonen2023InducedGrid} proved that, first of all, there is no qualitative difference between treewidth and $\alpha$-treewidth, and second a strong variant of the Grid Theorem holds for induced minors.
That is, here the only family obstructing treewidth is the family of grids.

\begin{theorem}[\cite{Korhonen2023InducedGrid}]\label{thm:inducedgrid}
There exists a function $f_{\ref{thm:inducedgrid}}(k,d)\in\mathcal{O}(k^{10}+2^{d^5})$ such that for a positive integer $k$, each graph $G$ with $\mathsf{tw}(G)\geq f_{\ref{thm:inducedgrid}}(k,\Delta(G))$ contains the $(k\times k)$-grid as an induced minor, where $\Delta(G)$ is the maximum degree of $G$.
\end{theorem}

Within the realm of induced substructures, a natural generalization of bounding the maximum degree would be bounding the maximum `induced' degree.
That means, as bounding the maximum degree by $d-1$ can be expressed as forbidding $K_{1,d}$ as a subgraph, we may instead forbid $K_{1,d}$ as an induced subgraph.
In this sense, we focus on $K_{1,d}$-free graphs in this paper.

As hinted at before, in the $K_{1,d}$-free case, treewidth is not an appropriate width parameter to consider.

While, due to \cref{thm:inducedgrid}, in the $K_{1,d}$-\textsl{subgraph} free setting, there is no qualitative difference between treewidth and $\alpha$-treewidth, $K_{1,d}$-free graphs permit arbitrarily large cliques and thus, a distinction between the two parameters emerges.
Hence, it seems reasonable to focus on tree-independence number as it allows for large cliques.
The \emph{tree-independence number} was defined independently by Yolov \cite{YolovMinorMatchingHypertreeWidth}, and Dallard, Milanič, and Štorgel \cite{dallard2021treewidthversuscliquenumber1}.
This modified width parameter measures the independence number of each bag in a tree-decomposition, instead of their cardinality and we denote\footnote{In \cite{dallard2021treewidthversuscliquenumber1}, the authors denoted the tree-independence number by $\mathsf{tree}$-$\alpha$.} it by $\alpha$-$\mathsf{tw}$.
In \cite{dallard2024treewidthversuscliquenumber4}, the authors conjectured an induced analogue of the Grid Theorem for $K_{1,d}$-free graphs, using the tree-independence number.

\begin{conjecture}[\cite{dallard2024treewidthversuscliquenumber4}]\label{conj:K1dfreeinducedgrid}
    There exists a function $f_{\ref{conj:K1dfreeinducedgrid}}:\mathbb{N}^2\rightarrow \mathbb{N}$ such that every $K_{1,d}$-free graph $G$ with $\alpha$-$\mathsf{tw}(G)\geq f_{\ref{conj:K1dfreeinducedgrid}}(k,d)$ contains the $(k\times k)$-grid as an induced minor.
\end{conjecture}

Over the past few years, various set of authors proved the above conjecture for several small planar graphs.
This includes tripods \cite{dallard2024treewidthversuscliquenumber4}, disjoint cycles with length at least $\ell$ \cite{ahn2024coarseerdhosposatheorem}, and wheels \cite{choi2025excludinginducedwheelminor}.
Also in \cite{chudnovsky2024treeindependencenumberiii}, the authors showed that under the same condition, one can find either a theta or a prism as an induced subgraph whenever $\alpha$-treewidth is large.
These results suggest strong evidence toward the ultimate goal: confirming \cref{conj:K1dfreeinducedgrid}.

In a similar manner, recently Gollin, Hatzel, and Wiederrecht \cite{gollin2025graphssimplestructuremaximal} extended \Cref{thm:inducedgrid} to a wider class of $K_{1,d}$-free graphs, but still with a considerably strong sparsity assumption.


\subsection{Our results}
The main result of this paper is that the exclusion of a ladder graph as an induced minor in $K_{1,d}$-free graphs results in bounded tree-independence number.

A \emph{$k$-ladder} is a graph obtained from two disjoint induced paths $P_1,P_2$ with $k$ vertices, by adding an edge between the $i$-th vertices of both paths for each $i\in[k]$.
We call $P_1$ and $P_2$ the \emph{rail paths}, and the edges between them the \emph{rungs}.
The \emph{$k$-skinny ladder} is the graph obtained from the $k$-ladder by subdividing each rung once.
See \cref{fig:ladder} for an illustration.

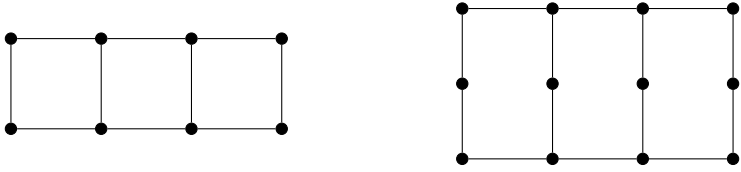
\begin{figure}
    \centering
    \begin{tikzpicture}
        \tikzset{enclosed/.style={draw, circle, inner sep=0pt, minimum size=.15cm,  fill=black}}
        \node[enclosed] (A) at (0,0) [] {};
        \node[enclosed] (B) at (1.2,0) [] {};
        \node[enclosed] (C) at (2.4,0) [] {};
        \node[enclosed] (D) at (3.6,0) [] {};
        \node[enclosed] (E) at (0,1.2) [] {};
        \node[enclosed] (F) at (1.2,1.2) [] {};
        \node[enclosed] (G) at (2.4,1.2) [] {};
        \node[enclosed] (H) at (3.6,1.2) [] {};
        \draw (A) -- (B) node[] () {};
        \draw (B) -- (C) node[] () {};
        \draw (C) -- (D) node[] () {};
        \draw (E) -- (F) node[] () {};
        \draw (F) -- (G) node[] () {};
        \draw (G) -- (H) node[] () {};
        \draw (A) -- (E) node[] () {};
        \draw (B) -- (F) node[] () {};
        \draw (C) -- (G) node[] () {};
        \draw (D) -- (H) node[] () {};
        \node[enclosed] (I) at (6,-0.4) [] {};
        \node[enclosed] (J) at (7.2,-0.4) [] {};
        \node[enclosed] (K) at (8.4,-0.4) [] {};
        \node[enclosed] (L) at (9.6,-0.4) [] {};
        \node[enclosed] (M) at (6,0.6) [] {};
        \node[enclosed] (N) at (7.2,0.6) [] {};
        \node[enclosed] (O) at (8.4,0.6) [] {};
        \node[enclosed] (P) at (9.6,0.6) [] {};
        \node[enclosed] (Q) at (6,1.6) [] {};
        \node[enclosed] (R) at (7.2,1.6) [] {};
        \node[enclosed] (S) at (8.4,1.6) [] {};
        \node[enclosed] (T) at (9.6,1.6) [] {};
        \draw (I) -- (J) node[] () {};
        \draw (J) -- (K) node[] () {};
        \draw (K) -- (L) node[] () {};
        \draw (I) -- (M) node[] () {};
        \draw (J) -- (N) node[] () {};
        \draw (K) -- (O) node[] () {};
        \draw (L) -- (P) node[] () {};
        \draw (Q) -- (M) node[] () {};
        \draw (R) -- (N) node[] () {};
        \draw (S) -- (O) node[] () {};
        \draw (T) -- (P) node[] () {};
        \draw (Q) -- (R) node[] () {};
        \draw (R) -- (S) node[] () {};
        \draw (S) -- (T) node[] () {};
    \end{tikzpicture}
    \caption{A $4$-ladder (left) and a $4$-skinny ladder (right)}
    \label{fig:ladder}
\end{figure}

\begin{theorem}\label{thm:mainresult}
    There exists a function $\tau:\mathbb{N}^2\rightarrow \mathbb{N}$ such that for all positive integers $k$ and $d\geq 2$ and every $K_{1,d}$-free graph $G$, either
    \begin{enumerate}
        \item $G$ contains the $k$-ladder as an induced minor, or
        \item $\alpha$-$\mathsf{tw}(G)\leq \tau(k,d)$.
    \end{enumerate}
\end{theorem}

Actually, we show a stronger statement, that is \cref{thm:mainresult} remains true even if \textsl{$k$-ladder} is replaced by \textsl{$k$-skinny ladder}.

\begin{theorem}\label{thm:themainresult}
    There exists a function $\tau:\mathbb{N}^2\rightarrow \mathbb{N}$ such that for all positive integers $k$ and $d\geq 2$ and every $K_{1,d}$-free graph $G$, either
    \begin{enumerate}
        \item $G$ contains the $k$-skinny ladder as an induced minor, or
        \item $\alpha$-$\mathsf{tw}(G)\leq \tau(k,d)$.
    \end{enumerate}
\end{theorem}

Indeed, our proof naturally gives \cref{thm:themainresult} due to the way we construct our induced ladder minor.
Within our proof we find two induced paths joined by a collection of pairwise vertex-disjoint and non-adjacent paths.
We then want to apply the Theorem of Erd\H{o}s and Szekeres to order these paths in a way that resembles the $k$-ladder.
As we are considering induced minors, this ``sorting'' process is unable to discard unwanted paths if they are of length $1$.
Hence, we need to guarantee that each paths between the two rail paths has length at least $2$, which then results in us finding the $k$-skinny ladder as an induced minor.

From the perspective of \Cref{conj:K1dfreeinducedgrid}, we may consider the $k$-ladder as the first floor of the $(k\times k)$-grid.
Indeed, we believe that our main theorem can be regarded as a first major starting point toward an induced grid theorem for $K_{1,d}$-free graphs.
Moreover, our proofs are constructive, and with the result of \cite{choi2025excludinginducedwheelminor}, they imply an algorithm that, for a given $K_{1,d}$-free graph $G$, either finds the $k$-ladder as an induced minor, or constructs a tree-decomposition of $G$ whose bags have independence number less than $\tau(k,d)$, in time $|V(G)|^{g(k,d)}$, where $g$ is a computable function.
\medskip

The strengthening from ladders to skinny ladders allows us to deduce simple and unified proofs for several known results in the area as discussed below.

In \cite{choi2025excludinginducedwheelminor}, the authors proved that every $K_{1,d}$-free graph excluding a wheel as an induced minor has bounded tree-independence number where the bound is given by $f_{\ref{thm:forbiddingthewheel}}(\ell,d)\in\mathcal{O}(d\ell^{11}+2^{(\max\{ \ell,d\})^5})$.
By using our method, we are able to considerably improve the function, with much a simpler proof.

Here the $\ell$-wheel, denoted by $W_{\ell}$ is the graph obtained from a cycle $C_{\ell}$ with $\ell$ vertices by introducing a new vertex and joining it to all vertices of $C_{\ell}$.

\begin{restatable}{theorem}{thmfindingwheel}\label{thm:findingwheel}
    For all positive integers $d$ and $\ell\geq 3$, if $G$ is a $K_{1,d}$-free graph with $\alpha$-$\mathsf{tw}(G)\geq8\ell (d-1)+8d-14$, then $G$ contains $W_\ell$ as an induced minor.
\end{restatable}

Moreover, we improve the result of Chudnovsky, Hajebi, and Trotignon regarding so-called \textsl{three path configurations} from \cite{chudnovsky2024treeindependencenumberiii} as follows:
A graph $G$ is called a \emph{theta} if it consists of two non-adjacent vertices $v$ and $w$ and three internally disjoint paths $P_1,P_2,P_3$ from $v$ to $w$, each of length at least two, such that there are no edges between the internal vertices of $P_i$ and $P_j$ for all distinct choices for $i,j\in[3]$.
We call $v$ and $w$ the \emph{ends} of $G$.
Let $\mathsf{Theta}$ denote the class of all thetas.
A \emph{$k$-long theta} is a theta such that each of the paths $P_1,P_2,P_3$ has length at least $k$.
For each $k\geq 2$, let $\mathsf{Theta}_k$ denote the class of $k$-long thetas.
See \cref{fig:longthetaprism} for an illustration.

\begin{figure}
    \centering
    \begin{tikzpicture}[scale=1.2]
    \pgfdeclarelayer{background}
    \pgfdeclarelayer{foreground}
    \pgfsetlayers{background,main,foreground}
    \begin{pgfonlayer}{background}
        \pgftext{\includegraphics[width=6cm]{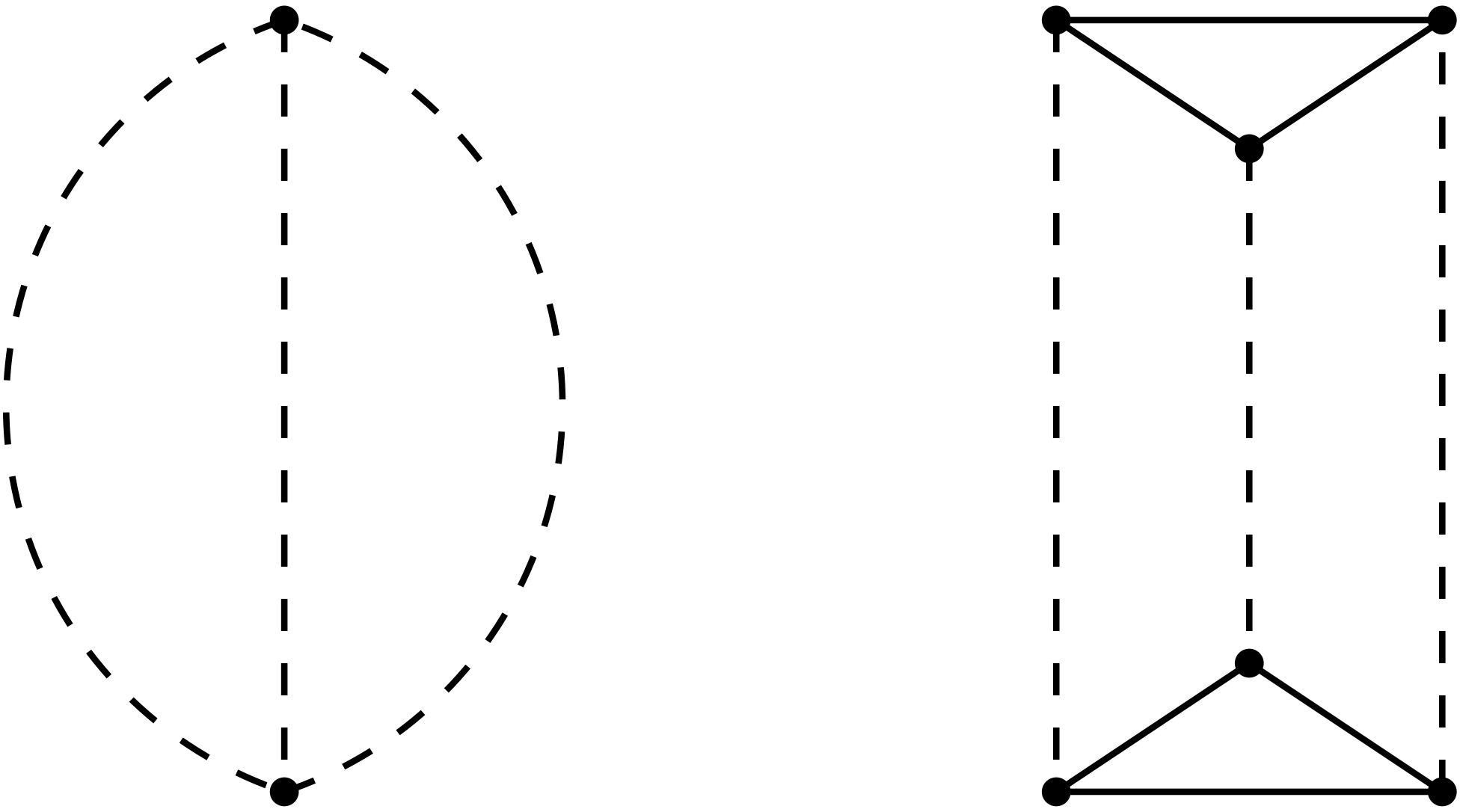}} at (C.center);
    \end{pgfonlayer}{background}
    \begin{pgfonlayer}{main}
        \node[]() at (-1.84,1.83) {$v$};
        \node[]() at (-1.84,-1.83) {$w$};
        \node[]() at (-3.24,0) {$P_1$};
        \node[]() at (-2.12,0) {$P_2$};
        \node[]() at (-1,0) {$P_3$};
        \node[]() at (1.2,1.83) {$v_1$};
        \node[]() at (1.2,-1.83) {$w_1$};
        \node[]() at (3,1.83) {$v_2$};
        \node[]() at (3,-1.83) {$w_2$};
        \node[]() at (2.13,1.33) {$v_3$};
        \node[]() at (2.13,-1.33) {$w_3$};
        \node[]() at (1.05,0) {$P_1$};
        \node[]() at (1.85,0) {$P_2$};
        \node[]() at (2.65,0) {$P_3$};
    \end{pgfonlayer}{main}
    \begin{pgfonlayer}{foreground}
    \end{pgfonlayer}{foreground}
\end{tikzpicture}
    \caption{A $k$-long theta (left) and a $k$-long prism (right). Dashed lines represent paths of length at least $k$.}
    \label{fig:longthetaprism}
\end{figure}
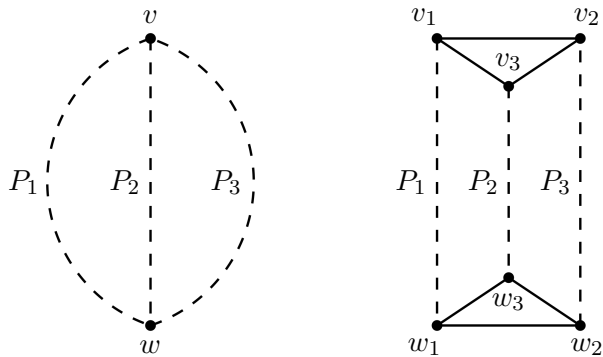

A graph $H$ is called a \emph{prism} if it consists of two triangles $\{v_1,v_2,v_3\}$, $\{w_1,w_2,w_3\}$, and three pairwise disjoint paths $P_1,P_2,P_3$, each with length at least $1$, such that for each $i\in [3]$, $P_i$ has $v_i$ and $w_i$ as endpoints, and for each distinct $i,j\in [3]$, $v_iv_j$ and $w_iw_j$ are the only edges between $P_i$ and $P_j$.
We call $\{v_1,v_2,v_3\}$ and $\{w_1,w_2,w_3\}$ the \emph{triangles} of $H$.
A \emph{generalized prism} is similar to a prism, but at most one of the path $P_i$ is allowed to have length $0$.
In this case, the length of the other two paths have to be at least $2$.
Let $\mathsf{Prism}$ denote the class of generalized prisms.
A \emph{$k$-long prism} is a prism such that each paths $P_1,P_2,P_3$ has length at least $k$.
For each $k\geq 2$, let $\mathsf{Prism}_k$ denote the class of $k$-long prism.

In \cite{chudnovsky2024treeindependencenumberiii}, Chudnovsky, Hajebi and Trotignon show that excluding all thetas and prisms as induced subgraphs in $K_{1,d}$-free graphs results in a class of bounded tree-independence number.
We generalize this result by showing that one can replace thetas and prisms in the theorem of Chudnovsky, Hajebi and Trotignon with $k$-long thetas and $k$-long prisms respectively.

\begin{restatable}{theorem}{thmklongthetaandprism}\label{thm:klongthetaandprism}
    For all positive integers $k,d$, the class of ($K_{1,d},\mathsf{Theta}_k,\mathsf{Prism}_k$)-free graphs has bounded tree-independence number.
\end{restatable}

\subsection{The problem of connecting vertex sets with disjoint paths}
We prove \Cref{thm:themainresult} by finding a ladder-like graph between a specific induced subgraph $H$ and an induced path $P$ where both $H$ and $P$ are chosen with respect to certain connectivity requirements.
The precise definition of this ladder-like graph is postponed to \Cref{sec:Ladderlikegraphs}.
The case where $H$ is a cycle is of particular interest.

The critical impediment for finding our ladder-like graphs comes from the absence of an induced version of Menger's theorem (see in particular \cite{Nguyen2025CounterexampletocoarseMenger,NguyenSS2025Asymptotic}).
There are various proofs of the Grid Theorem, e.g. \cite{GraphMinor5,Chekuri2016Polynomialgridtheorem}, but a unifying feature of all those proofs is a fundamental reliance on Menger's Theorem.
This key theorem yields bridges between sets of vertices that cannot be separated easily, thus guarantees the existence of many (intertwined) paths that eventually give rise to the grid minor.
Unfortunately, if we want to find some graph as an induced minor, or induced subgraph, we can no longer utilize the benefit of this powerful theorem.
However, approximate versions of Menger's Theorem for the existence of pairwise non-adjacent paths might still hold.
There are many conjectured variants of Menger's Theorem for the induced setting suggesting such analogues, and the following is an attempt of formulating a desirable version of such a theorem for the setting of this paper.

\begin{conjecture}\label{conj:inducedMenger}
    There exists a function $f_{\ref{conj:inducedMenger}}:\mathbb{N}^2\rightarrow \mathbb{N}$ such that for every $K_{1,d}$-free graph $G$ and two disjoint subsets of vertices $A,B\subseteq V(G)$, there exists either
    \begin{enumerate}
        \item $k$ pairwise non-adjacent $A$-$B$ paths, or
        \item a separator $S$ between $A$ and $B$ with $\alpha(S)<f_{\ref{conj:inducedMenger}}(k,d)$.
    \end{enumerate}
\end{conjecture}

There have been several results on induced Menger's theorem on graphs with bounded maximum degree. In \cite{hendrey2024onaninducedversionofMengerstheorem} and \cite{gartland2023inducedversionsmengerstheorem}, two sets of authors independently showed similar versions of an induced Menger's theorem on graphs with bounded degree.
But beyond that, only few successful attempts have been reported \cite{Albrechtsen2024AMengertypetheoremfortwoinducedpaths}.
Some braver conjectures in this area have even been refuted \cite{Nguyen2025CounterexampletocoarseMenger,NguyenSS2025Asymptotic}.

In light of these findings, it appears sensible to ask whether a weaker variant would be able to fill the same role in a potential proof of a grid theorem for induced minors.
In the following we loosely refer to any theorem which fits into the mold below as a \emph{Menger-like theorem}.
Here ``Menger-like'' means that there either exists a large collection of pairwise disjoint and non-adjacent paths between two subgraphs satisfying some property depending on two initially given sets of vertices $X$ and $Y$ or guarantees the existence of a vertex set $S$ of small independence number such that $G-S$ has no path between $X$ and $Y$.
See \Cref{thm:somethingpath} for an example of such a Menger-like theorem. 

Indeed, we introduce \Cref{thm:somethingpath} to address the lack of an (approximate) duality theorem for non-adjacent paths and separators of small independence number.
In \Cref{thm:somethingpath}, instead of finding  $k$ non-adjacent paths between two fixed vertex sets, we fix only one of those sets and allow the other one to be replaced by another, while maintaining that the new vertex set induces a path.
Notice that allowing this slack is strong enough for constructing our ladder, but anything beyond the connectivity between \textsl{two} objects seems to be out of reach of this method for the time being.
\medskip

Another major obstacle for analyzing $K_{1,d}$-free graphs is that we do not have much control over the set of vertices or edges at distance at most $k$ from some fixed set of vertices or edges.
For $v\in V(G)$, let $N_k[v]$ denote the set of vertices at distance at most $k$ from $v$.
If the graph $G$ has maximum degree $\Delta$, i.e. if we have $|N_1[v]|\leq \Delta+1$, we also have $|N_k[v]|\leq \Delta((\Delta-1)^k-1)+1$.
However, if $G$ is $K_{1,d}$-free, i.e. if we have $\alpha(N_1[v])\leq d-1$, even $\alpha(N_2[v])$ can be unbounded.
For example, Let $G$ be a graph on $2n$ vertices consisting of an independent set $\{v_1,v_2,\cdots,v_n\}$ and a clique $\{w_1,w_2,\cdots,w_n\}$ such that $v_i$ and $w_j$ are adjacent if and only if $i=j$.
Then $G$ is $K_{1,3}$-free, but $\alpha(N_2[w_i])=n$ can be arbitrarily large.
In fact, this large clique with personal neighbors is essentially the graph which destroys the advantage of $K_{1,d}$-freeness, as we contract the clique, we get a vertex whose neighbor has large independence number.

The proof of \Cref{thm:somethingpath} is primarily concerned with the construction of paths whose intersections are sufficiently well-separated.
This ensures that, upon contracting the junctions where paths intersect, the independence number of resulting vertices remains bounded.

\subsection{Organization of the paper}
\Cref{sec:Prelim} introduces the notations and previous results on tree-independence number.
In \Cref{sec:Ladderlikegraphs}, we define several ladder-like graphs that appear in the proof of our main theorem.
Then, in \Cref{sec:Cleaning}, we prove that our relaxed ladder-like graphs, if chosen large enough, contain the (skinny) $k$-ladder as an induced minor.
In \Cref{sec:DividingPath}, we show how to find an induced cycle and an induced path such that every separator between them has large independence number within a graph of large $\alpha$-treewidth.
Finally, in \Cref{sec:MainTheorem}, we find $k$ disjoint paths between a fixed subgraph and an induced path, where the induced path may be changed to another induced path.
With the results of the two previous sections, we are now able to find the desired (skinny) $k$-ladder as an induced minor in a $K_{1,d}$-free graph with large tree-independence number.
After that, in \Cref{sec:Substructures}, we deduce the results from \cite{choi2025excludinginducedwheelminor} and \cite{chudnovsky2024treeindependencenumberiii} using our methods.
Moreover, we provide a generalization of a recent result of the Erd\H{o}s-P{\'o}sa property of long and pairwise non-adjacent induced cycles due to Ahn, Gollin, Huynh, and Kwon \cite{ahn2024coarseerdhosposatheorem,ahn2025coarseerdhosposatheorem}.

We conclude the paper in \Cref{sec:Conclusion} by discussing possible next steps towards \Cref{conj:K1dfreeinducedgrid}.

\section{Preliminaries}\label{sec:Prelim}
For a positive integer $a$, we denote by $[a]$ the set $\{1,\cdots,a\}$.

All graphs in this paper are finite and simple.
We mostly follow the notations and conventions from \cite{diestel2016graph}.

For a graph $G$ and a set $X\subseteq V(G)$, we denote by $G-X$ the graph obtained from $G$ by deleting all vertices in $X$.
Also, we denote by $G[X]$ the graph obtained from $G$ by deleting all vertices not in $X$, i.e.\@ $G[X]$ is the subgraph of $G$ \emph{induced by $X$}.

Here is a basic observation on $K_{1,d}$-free graphs we use repeatedly.

\begin{observation}\label{lem:nbdonpath}
    Let $d$ be a positive integer, let $G$ be a $K_{1,d}$-free graph, and let $A$ be an induced subgraph of $G$ that is either a path or a cycle of length at least $4$.
    Then for each $v\in V(G)\setminus V(A)$, we have $|N(v)\cap V(A)|\leq 2(d-1)$.
\end{observation}

We also use following observation on the transitivity of the induced minor relation.

\begin{observation}\label{lem:inducedminorofinducedminor}
    Let $G_1,G_2,G_3$ be graphs.
    If $G_1$ is an induced minor of $G_2$ and $G_2$ is an induced minor of $G_3$, then $G_1$ is an induced minor of $G_3$.
\end{observation}

\emph{Subdividing} an edge $uv\in E(G)$ is defined as follows: we remove the edge $uv$ from $E(G)$, introduce a new vertex $w$, and add edges $uw$ and $vw$ to $E(G)$. 

The \emph{neighborhood} of $X\subseteq V(G)$ is defined as the set $N_G(X)\coloneqq \{v\in V(G)\setminus X\mid \text{ there exists } x\in X \text{ such that }vx\in E(G)\}$.
We often say that $v\in N_G(X)$ is a \emph{neighbor} of $X$.
The \emph{closed neighborhood} of $X$ is defined by $N(X)\cup X$, and denoted by $N_G[X]$.
If $X=\{x\}$, we simply write $N_G(x)$ and $N_G[x]$, instead of $N_G(\{x\})$ and $N_G[\{x\}]$, respectively.
Also, we may omit the subscript $G$ if there is no ambiguity.
Similarly, we can define the \emph{distance-$\ell$ neighborhood} of $x\in V(G)$ by the set of vertices whose distance from $x$ is at most $\ell$, which is denoted by $N_\ell[v]$.

For $X,Y\subseteq V(G)$, we say that $X$ and $Y$ are \emph{adjacent} if $N[X]\cap Y\neq\emptyset$, and otherwise, we say that $X$ and $Y$ are \emph{non-adjacent}.

Let $P=v_1v_2\cdots v_{n-1}v_n$ be a path.
We say that $v_1$ and $v_n$ are the \emph{endpoints} of $P$, and if $n\geq 2$, we say that $v_2$ and $v_{n-1}$ are the \emph{second-endpoints} of $P$.
Vertices of $P$ that are not endpoints are said to be \emph{internal vertices}.
If $n=1$, i.e. if the length of $P$ is $0$, we say that $P$ does not have a second-endpoint.
To avoid cumbersome case distinctions, we sometimes say that `let $v_1$ be one endpoint of $P$ and let $v_n$ be the other endpoint of $P$' even if $v_1=v_n$.

For the sake of simplicity, we think of a path as a graph with an `order' on its vertices, in the sense that when we have $X=\{x_1,x_2,\cdots,x_k\}\subseteq V(P)$, with $x_i=v_{i'}$ for $i\in [k]$, we always assume that $i'<j'$ if $i<j$.
In this case, we say that $v_1$ is the \emph{first endpoint}, and $v_n$ is the \emph{last endpoint} of $P$.
Also, each subpath of a path is considered to have the same `order'.

For $1\leq i\leq j\leq n$, we denote by $v_iPv_j$ the subpath of $P$ between $v_i$ and $v_j$.
Also, we denote by $\mathring{v}_iP\mathring{v}_j$ the subpath of $P$ defined by $v_iPv_j-\{v_i,v_j\}$.
For a subpath $Q$ of $P$ and $v\in V(P)$, we say that a subpath $Q'$ is \emph{obtained from extending $Q$ to contain $v$} if $Q'$ is the minimal subpath which contains both $v$ and $Q'$.
Lastly, We say that vertex sets $\{A_i\}_{i\in [k]}$ \emph{appears in order along $P$} if $\emptyset \neq A_i\subseteq V(P)$ for each $i$ and there exists integers $1=x_1<x_2<\cdots<x_{k+1}=n+1$ such that $V(A_i)\subseteq V(p_{x_i}Pp_{x_{i+1}-1})$.
Note that this implies that $A_i$'s are disjoint.
We also say that vertices $\{a_i\}_{i\in [k]}$ \emph{appears in order along $P$} if $a_i\in V(P)$ and $\{\{a_i\}\}_{i\in [k]}$ appears in order along $P$.

Let $X,Y,Z\subseteq V(G)$ be vertex sets.
We say that a path $P$ is a \emph{path from $X$ to $Y$ through $Z$} if one endpoint of $P$ is in $X$, the other endpoint is in $Y$, and all other internal vertices are in $Z\setminus (X\cup Y)$.
Note that the endpoints of $P$ may not be in $Z$.
If $Z=V(G)$, we may omit $Z$.
We also say that \emph{$x\in X$ is reachable from $Y$ through $Z$} if there is a path from $X$ to $Y$ through $Z$ such that one endpoint is $x$.
Equivalently, $x\in X$ is reachable from $Y$ through $Z$ if there is a component $C$ of $(Y\cup Z)-X$ that is adjacent to $x$ and $V(C)\cap Y\neq \varnothing$.
We say that $S\subseteq V(G)$ is a \emph{separator} between $X$ and $Y$ if every vertices in $X$ is not reachable from $Y$ through $V(G)\setminus S$.
Here, we allow $S\cap X\neq \emptyset$ or $S\cap Y\neq \emptyset$.

\subsection{tree-independence number}
For a graph $G$, a \emph{tree-decomposition} is a pair $\mathcal{T}=(T,\beta)$ which consists of a tree $T$ and a map $\beta\colon V(T)\rightarrow 2^{V(G)}$ satisfying the following conditions: for each vertex $v\in V(G)$, there exists some $t\in V(T)$ such that $v\in \beta(t)$, for each edge $uv\in E(G)$, there exists some $t\in V(T)$ such that $\{u,v\}\subseteq \beta(t)$, and for each vertex $v\in V(G)$ the graph $T[\{t\in V(T)\mid v\in \beta(t)\}]$ is connected.
The independence number of $\mathcal{T}$, denoted by $\alpha(\mathcal{T})$, is defined by $\max_{t\in V(T)}\alpha(G[\beta(t)])$.
The tree-independence number of $G$, denoted by $\atw(G)$, is defined as $\min_{\mathcal{T}}\alpha(\mathcal{T})$, where minimum is taken over all tree-decompositions $\mathcal{T}$ for $G$.

Adler \cite{IsoldeThesis} defined bramble-like objects for generalized width parameters defined by modifying treewidth, and proved a duality theorem between the generalized bramble number and corresponding width parameter.
This bramble-like object corresponding to the tree-independence number is the key concept toward proving our main theorem.

A \emph{strong bramble} in $G$ is a collection $\mathcal{B}=\{B_1,\dots,B_n\}$ of vertex subsets $B_i\subseteq V(G)$ such that $G[B_i]$ is connected for all $i\in [n]$, and $B_i\cap B_j\neq\emptyset$ for all $i,j\in [n]$.
A \emph{hitting set} of $\mathcal{B}$ is a set $X$ such that $X\cap B_i\neq \emptyset$ for all $i\in [n]$.
The $\alpha$-order of $\mathcal{B}$ is defined by $\min\{ \alpha(X) \mid X\subseteq V(G)\text{ is a hitting set of }\mathcal{B}\}$.

A short proof of this bramble-duality for tree-independence number can be found in \cite{choi2025excludinginducedwheelminor}.

\begin{theorem}[\cite{IsoldeThesis,choi2025excludinginducedwheelminor}]\label{thm:alphaduality}
    For every graph $G$ and a positive integer $k$, the following statements hold.
    \begin{enumerate}
        \item If $G$ has a strong bramble of $\alpha$-order $k$, then $\alpha\text{-}\mathsf{tw}(G)\geq k$.
        \item If $\alpha\text{-}\mathsf{tw}(G)\geq 4k-2$ then $G$ has a strong bramble of $\alpha$-order at least $k$.
    \end{enumerate}
\end{theorem}

Note that we use the notion of strong bramble that requires each pair of bramble elements to intersect, instead of the usual bramble notion where each pair of bramble elements either intersect or adjacent.

The core use of (strong) brambles in this paper, besides them witnessing large tree-independence number, is that they naturally give rise to a specific induced path which we use as the base for constructing the rail paths of our (skinny) ladders.

\begin{lemma}[\cite{choi2025excludinginducedwheelminor}]\label{lemma:dominatingpath}
    Let $G$ have a strong bramble $\mathcal{B}$. Then there is an induced path $P$ such that $N[P]\cap B\neq \emptyset$ for all $B\in \mathcal{B}$.
\end{lemma}

Notice that, in $K_{1,d}$-free graphs, if $\mathcal{B}$ is a strong bramble of $\alpha$-order at least $dw$, then the path $P$ found by \cref{lemma:dominatingpath} must have at least $w$ vertices.

In \cite{choi2025excludinginducedwheelminor}, the authors link the endpoints of such path to form a cycle.
In \Cref{sec:DividingPath}, we further modify this method to form a cycle using only a part of the path.

\section{Ladder-like graphs}\label{sec:Ladderlikegraphs}

The primary goal of this paper is finding a ladder as an induced minor.
As a first step, we define several ladder-like graphs that occur as intermediate steps.

Let $k$ be a positive integer.
A \emph{$k$-shuffled rope ladder} is a graph consisting of induced paths $P_1,P_2,\Phi^1,\Phi^2,\cdots,\Phi^k$ (possibly with length $0$), where $\phi^i_1$ and $\phi^i_2$ are endpoints of $\Phi^i$, such that
\begin{itemize}
    \item $P_1,P_2,\Phi^1,\cdots,\Phi^k$ are mutually disjoint,
    \item $P_1$ and $P_2$ are non-adjacent,
    \item $\Phi^1,\cdots,\Phi^k$ are mutually non-adjacent, and
    \item $N(\Phi^i)\cap V(P_j)= N(\phi_j^i)\cap V(P_j)\neq \emptyset$ for each $i\in [k]$ and $j\in [2]$.
\end{itemize}

Similar to the $k$-ladder, we refer to $P_1$ and $P_2$ as \emph{rail paths}, and $\Phi^i$'s as \emph{rung paths} of a $k$-shuffled rope ladder.
Note that unlike in ladders, rung paths and rail paths are disjoint, but they must still be adjacent.
See \cref{fig:shuffledropeladder} for an illustration.

A \emph{$k$-rope ladder} is a shuffled rope ladder with the following additional condition: 
$\{N(\phi_j^i)\cap V(P_j)\}_{i\in [k]}$ appears in order along $P_j$ for each $j\in\{1,2\}$.

\begin{figure}
    \centering
    \begin{tikzpicture}[scale=1.15]
    \pgfdeclarelayer{background}
    \pgfdeclarelayer{foreground}
    \pgfsetlayers{background,main,foreground}
    \begin{pgfonlayer}{background}
        \pgftext{\includegraphics[width=6cm]{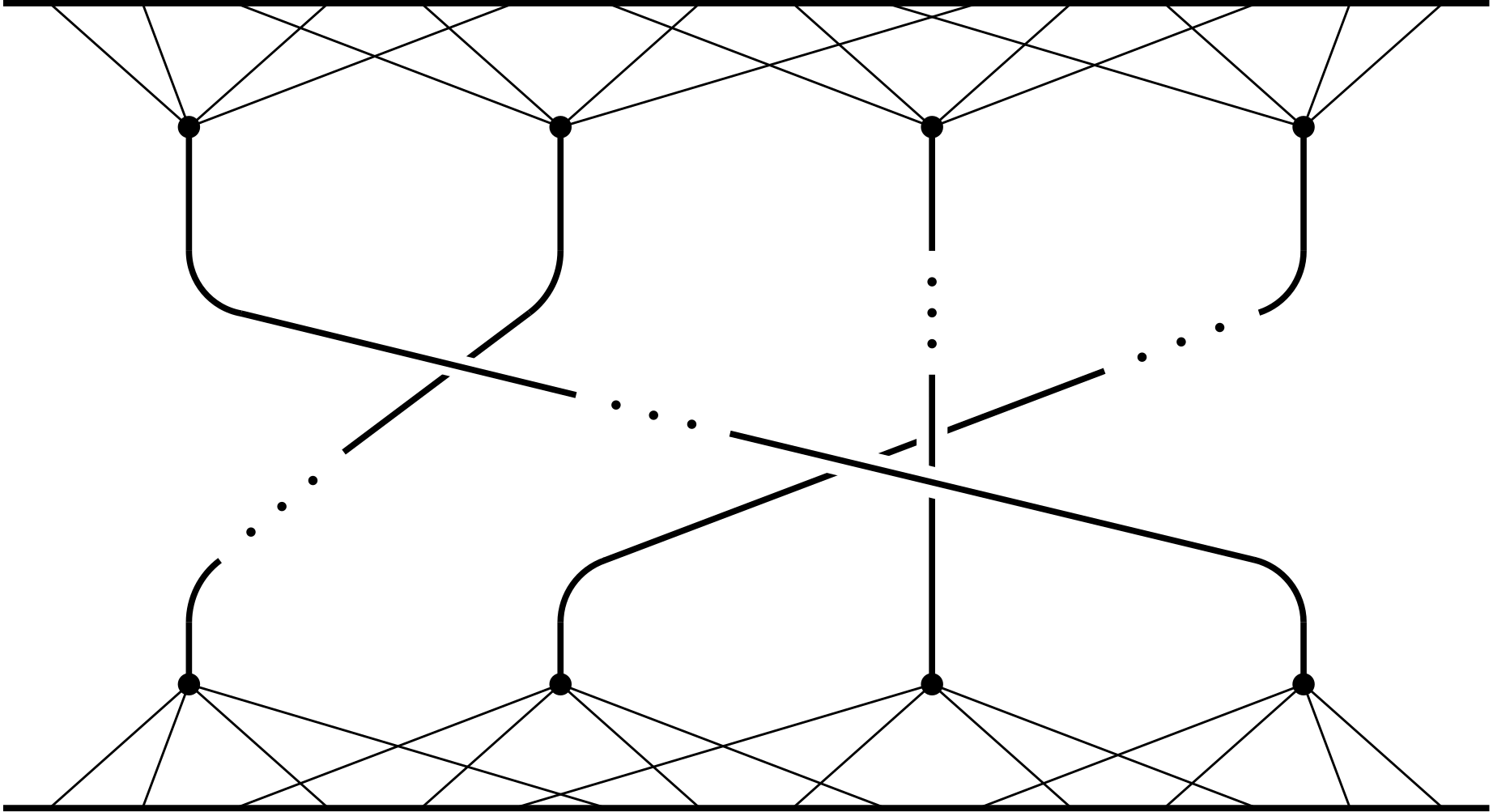}} at (C.center);
    \end{pgfonlayer}{background}
    \begin{pgfonlayer}{main}
        \node[]() at (-3.2,1.6) {\large $P_1$};
        \node[]() at (-3.2,-1.65) {\large $P_2$};
        \node[]() at (-2,1) {$\phi_1^1$};
        \node[]() at (-0.5,1) {$\phi_1^2$};
        \node[]() at (1,1) {$\phi_1^3$};
        \node[]() at (2.5,1) {$\phi_1^4$};
        \node[]() at (-2.47,0.55) {\large $\Phi^1$};
        \node[]() at (-1,0.55) {\large $\Phi^2$};
        \node[]() at (0.47,0.55) {\large $\Phi^3$};
        \node[]() at (1.97,0.55) {\large $\Phi^4$};
        \node[]() at (-2,-0.95) {$\phi_2^2$};
        \node[]() at (-0.5,-0.95) {$\phi_2^4$};
        \node[]() at (1,-0.95) {$\phi_2^3$};
        \node[]() at (2.5,-0.95) {$\phi_2^1$};
    \end{pgfonlayer}{main}
    \begin{pgfonlayer}{foreground}
    \end{pgfonlayer}{foreground}
    \end{tikzpicture}
    \hspace{0.25cm}
    \begin{tikzpicture}[scale=1.25]
    \pgfdeclarelayer{background}
    \pgfdeclarelayer{foreground}
    \pgfsetlayers{background,main,foreground}
    \begin{pgfonlayer}{background}
        \pgftext{\includegraphics[width=6cm]{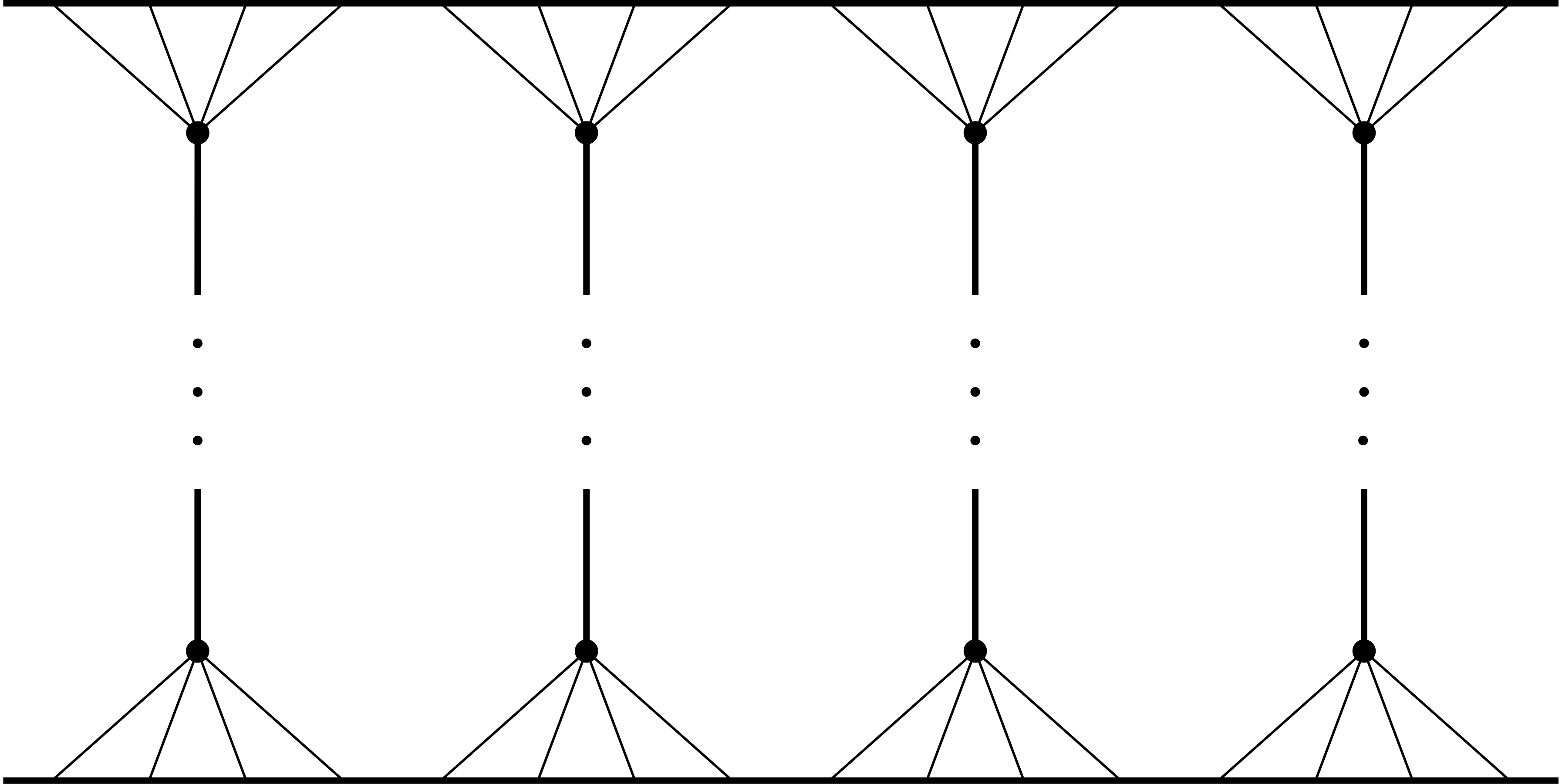}} at (C.center);
    \end{pgfonlayer}{background}
    \begin{pgfonlayer}{main}
        \node[]() at (-3.2,1.45) {$P_1$};
        \node[]() at (-3.2,-1.55) {$P_2$};
        \node[]() at (-2,0.9) {$\phi_1^1$};
        \node[]() at (-0.5,0.9) {$\phi_1^2$};
        \node[]() at (1,0.9) {$\phi_1^3$};
        \node[]() at (2.5,0.9) {$\phi_1^4$};
        \node[]() at (-2,-0.85) {$\phi_2^1$};
        \node[]() at (-0.5,-0.85) {$\phi_2^2$};
        \node[]() at (1,-0.85) {$\phi_2^3$};
        \node[]() at (2.5,-0.85) {$\phi_2^4$};
        \node[]() at (-2.47,0) {\large $\Phi^1$};
        \node[]() at (-1,0) {\large $\Phi^2$};
        \node[]() at (0.47,0) {\large $\Phi^3$};
        \node[]() at (1.97,0) {\large $\Phi^4$};
    \end{pgfonlayer}{main}
    \begin{pgfonlayer}{foreground}
    \end{pgfonlayer}{foreground}
    \end{tikzpicture}
    \caption{A $4$-shuffled rope ladder (left) and a $4$-rope ladder (right).}
    \label{fig:shuffledropeladder}
\end{figure}
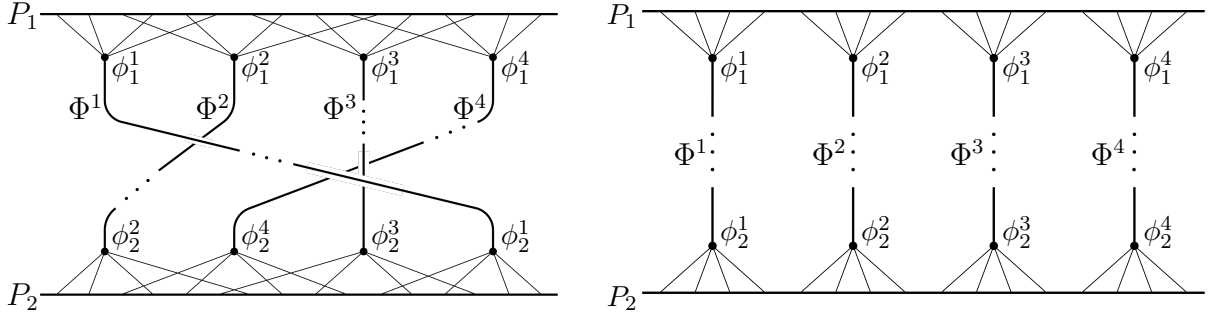

In \Cref{sec:Cleaning}, we show that for fixed $k$, there is some large $k'$ such that every $K_{1,d}$-free $k'$-shuffled rope ladder contains a $k$-rope ladder as an induced subgraph (see \cref{lem:cleanshuffledropeladder}).
It is easy to see that a $k$-rope ladder contains the $k$-skinny ladder as an induced minor.
Therefore, to show that a $K_{1,d}$-free graph $G$ contains the $k$-skinny ladder as an induced minor, it is enough to find a large shuffled rope ladder as an induced subgraph.
\medskip

We may further generalize shuffled rope ladders by replacing the rail paths with other subgraphs.
In this paper, we only consider the case where only one of the rail path is replaced by another graph $H$.
Indeed, we only need the case where $H$ is an induced cycle, but some of our proofs hold even for for general $H$.

A \emph{$k$-$H$-rope ladder}, is a graph consisting of any graph $H$ and induced paths $P,\Phi^1,\Phi^2,\cdots,\Phi^k$, where $\phi_1^i$ and $\phi_2^i$ denote the endpoints of $\Phi^i$, such that
\begin{itemize}
    \item $H,P,\Phi^1,\cdots,\Phi^k$ are mutually disjoint,
    \item $H$ and $P$ are non-adjacent,
    \item $\Phi^1,\cdots,\Phi^k$ are mutually non-adjacent,
    \item $V(H)\cap N(\Phi^i)=V(H)\cap N(\phi_1^i)\neq \emptyset$ for each $i\in [k]$, and
    \item $V(P)\cap N(\Phi^i)=V(P)\cap N(\phi_2^i)\neq \emptyset$ for each $i\in [k]$.
\end{itemize}
We say that $H$ is a \emph{rail subgraph}, $P$ is a \emph{rail path}, and $\Phi^i$'s are rung paths of a $k$-$H$-rope ladder.

Note that if $H$ is an induced path, above definition is equal to the $k$-rope ladder.
When $H$ is an induced cycle, we denote this graph as a \emph{$k$-cycle rope ladder}.
In this case, we say that $C$ is a \emph{rail cycle} of a $k$-cycle rope ladder.
See \cref{fig:CycleRopeladder} for an illustration.

Notice that we drop the adjective ``\textsl{shuffled}'' in our definition for $k$-$H$-rope ladders.
That is, for any path $P$, a $k$-$P$-rope ladder is a $k$-shuffled rope ladder.
This is for two reasons, first for convenience and second because in the generality of $H$ we loose the natural order of the vertices provided by a path.

\begin{figure}
    \centering
    \begin{tikzpicture}[scale=1.5]
    \pgfdeclarelayer{background}
    \pgfdeclarelayer{foreground}
    \pgfsetlayers{background,main,foreground}
    \begin{pgfonlayer}{background}
        \pgftext{\includegraphics[width=6cm]{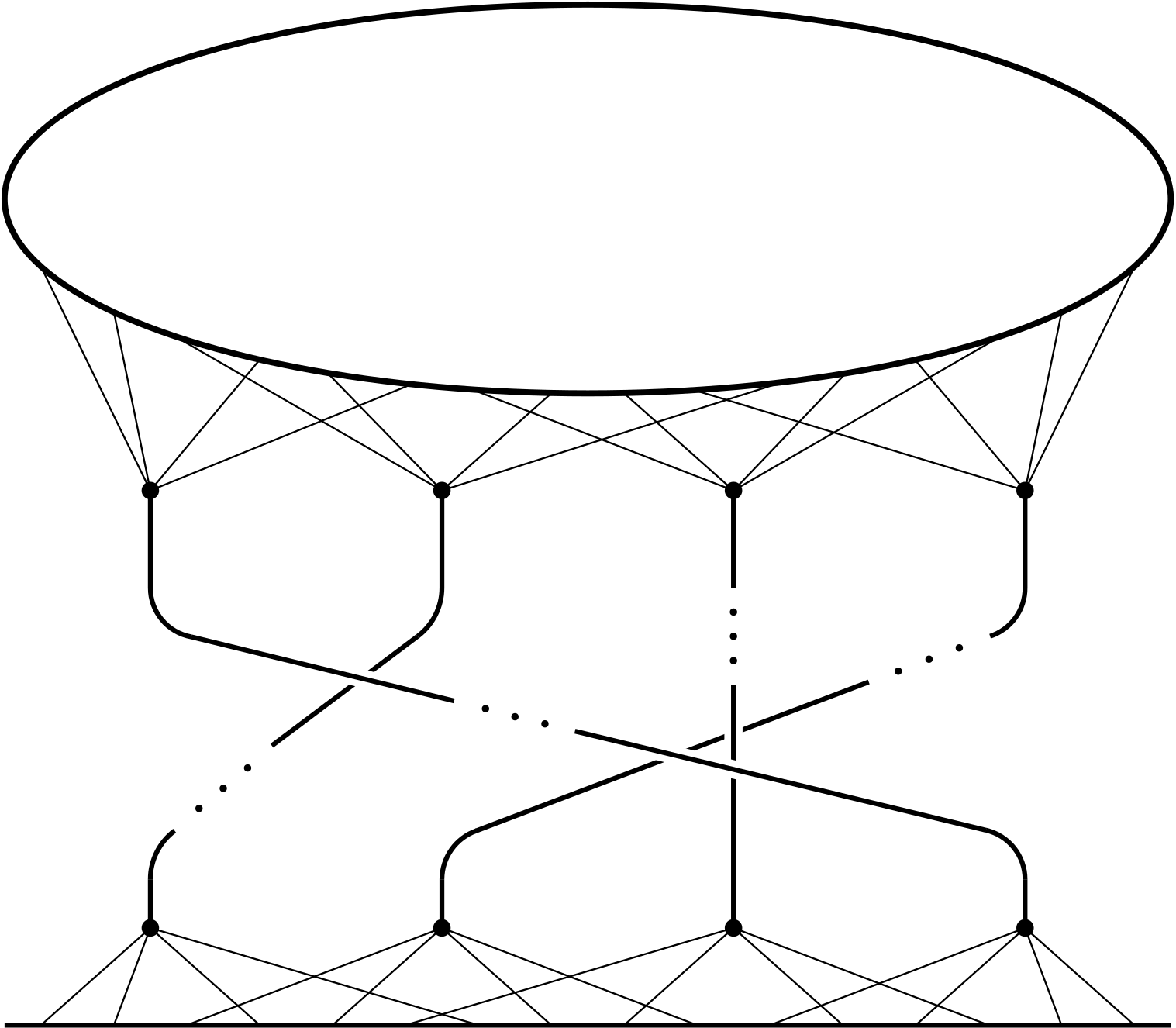}} at (C.center);
    \end{pgfonlayer}{background}
    \begin{pgfonlayer}{main}
        \node[]() at (-3.2,1.63) {$C$};
        \node[]() at (-3.2,-2.6) {$P$};
        \node[]() at (-2.05,0) {$\phi_1^1$};
        \node[]() at (-0.55,0) {$\phi_1^2$};
        \node[]() at (0.95,0) {$\phi_1^3$};
        \node[]() at (2.45,0) {$\phi_1^4$};
        \node[]() at (-2.05,-1.93) {$\phi_2^2$};
        \node[]() at (-0.55,-1.93) {$\phi_2^4$};
        \node[]() at (0.95,-1.93) {$\phi_2^3$};
        \node[]() at (2.45,-1.93) {$\phi_2^1$};
        \node[]() at (-2.45,-0.4) {\large $\Phi^1$};
        \node[]() at (-0.96,-0.4) {\large $\Phi^2$};
        \node[]() at (0.54,-0.4) {\large $\Phi^3$};
        \node[]() at (2.03,-0.4) {\large $\Phi^4$};
    \end{pgfonlayer}{main}
    \begin{pgfonlayer}{foreground}
    \end{pgfonlayer}{foreground}
\end{tikzpicture}
    \caption{A $4$-cycle rope ladder.}
    \label{fig:CycleRopeladder}
\end{figure}
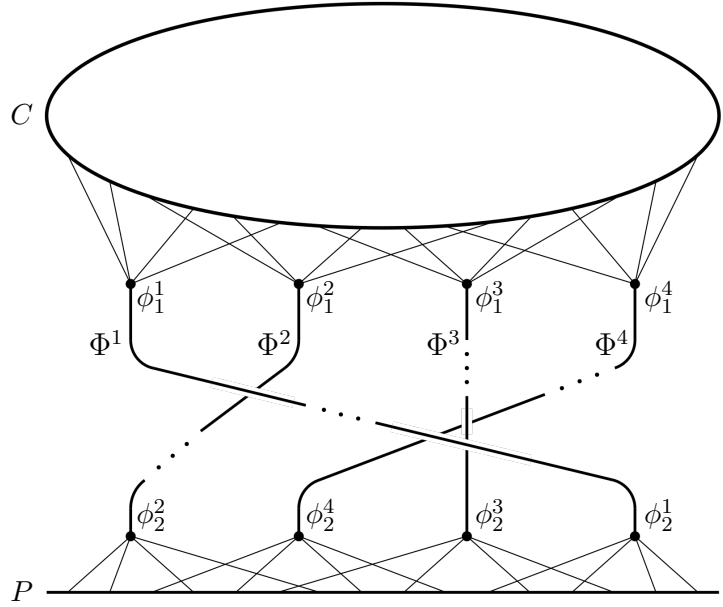

\section{Cleaning a shuffled rope ladder}\label{sec:Cleaning}

In this section, we prove that for each $k$, there is some large $k'$ such that a $K_{1,d}$-free $k'$-shuffled rope ladder contains the $k$-rope ladder as an induced subgraph.
As a first step, we prove a lemma that `cleans' one side of a shuffled rope ladder.

\begin{lemma}\label{cleanoneside}
    Let $k$ and $d\geq 2$ be positive integers.
    Let $G$ be a $K_{1,d}$-free graph with $V(G)=V(P)\cup A$, where $P$ is an induced path and $A=N(P)$ is an independent set of cardinality at least $\rho(k,d)\coloneqq (2(d-1)^2(2d-1)^{2d-2})^{k-1}$.
    Then we can find $B=\{b_1,b_2,\cdots,b_k\}\subseteq A$ and a subpath $Q$ of $P$, such that $\{N(b_i)\cap V(Q)\}_{i\in [k]}$ appears in order along $Q$.
\end{lemma}

\begin{proof}
Let $x$ be an endpoint of $P$.
Throughout the proof, we assume that the `order' of $P$ is fixed so that $x$ is the first vertex.
Let $A=\{a_1,a_2,\cdots,a_{\rho(k,d)}\}$, and let $N(a_i)=\{v_{1}^i,v_{2}^i,$ $\cdots,v_{\deg(a_i)}^i\}\subseteq V(P)$ be the neighbors of $a_i$, where the subscripts of the $v_j^i$s are given in the order according to their appearance on $P$.
By \Cref{lem:nbdonpath}, we have $\deg(a_i)\leq 2d-2$ for all $i$.
Let us name the vertices of $A$ so that $v_1^1$ appears first on $P$ among the $v_j^i$'s.
We use induction on $k$ to show a strengthening of our statement:
We additionally demand $a_1\in B$ and $v_1^1\in V(Q)$.

If $k=1$, we can choose $B=A$ and $Q=P$.
Now assume that $k\geq 2$ and that the statement holds for $k-1$.
Since $G$ is $K_{1,d}$-free, each vertex in $P$ is adjacent to at most $d-1$ vertices in $A$.
Hence, we may find $A'\subseteq A$ such that
\begin{itemize}
    \item $a_1\in A'$,
    \item $|A'|\geq |A|/(2(d-1)^2)\geq(2d-1)^{2d-2}\rho(k-1,d)$, and
    \item for each $a,a'\in A'$, $N(a)\cap N(a')=\emptyset$.
\end{itemize}
We may rename the vertices in $A'$ so that $v_{1}^i$ appears before $v_{1}^{i+1}$ in $P$, for all $i$.
Note that $a_1$ does not change by this renaming.
We define $\mathsf{gap}(x,y)\coloneqq |V(\mathring{x}P\mathring{y})\cap N(A')|$.
One may consider $\mathsf{gap}$ as a distance between $x$ and $y$, only counting the vertices that have a neighbor in $A'$.
We proceed with a case distinction on $\mathsf{gap}(v_j^1,v_{j+1}^1)$.
See \cref{fig:cleancase1} for an illustration.

\medskip

\textbf{Case 1:} \textsl{There is some $j\in [2d-3]$ such that $\mathsf{gap}(v_j^1,v_{j+1}^1)\geq (2d-2)(2d-1)^{j-1}\rho(k-1,d)$.}
Choose $j$ the be minimal with respect to the property above.
Then, since any vertex in $A'$ can have at most $2d-2$ neighbors in $\mathring{v}_{j}^1P\mathring{v}_{j+1}^1$, we have
\begin{align*}
    |(A'\setminus \{a_1\})\cap N(\mathring{v}_{j}^1P\mathring{v}_{j+1}^1)|\geq ~& \mathsf{gap}(v_j^1,v_{j+1}^1)/(2d-2)\\
    \geq ~& (2d-1)^{j-1}\rho(k-1,d).
\end{align*}

On the other hand, by the minimality of $j$, we have
\begin{align*}
|(A'\setminus\{a_1\}) \cap N(\mathring{v}_{1}^1P\mathring{v}_{j}^1)| \leq ~& \sum_{h\in [j-1]} \mathsf{gap}(v_h^1,v_{h+1}^1) \\
< ~& \big(\sum_{h\in[j-1]}(2d-2)(2d-1)^{h-1}\big)\rho(k-1,d)\\
= ~&((2d-1)^{j-1}-1)\rho(k-1,d).
\end{align*}

Let $A'' \coloneqq (A'\setminus\{a_1\})\cap (N(\mathring{v}_{j}^1P\mathring{v}_{j+1}^1)\setminus N(\mathring{v}_{1}^1P\mathring{v}_{j}^1))$ be the set of all members of $A'\setminus \{a_1\}$ that have a neighbor in $\mathring{v}_{j}^1P\mathring{v}_{j+1}^1$ but have no neighbor in $\mathring{v}_{1}^1P\mathring{v}_{j}^1$.
Then $|A''|\geq \rho(k-1,d)$, so we can use the induction hypothesis to obtain $B'\subseteq A''$ and a subpath $Q'$ of $P$.
Now let $B \coloneqq B'\cup \{a_1\}$ and let $Q$ be the minimal subpath of $P$ containing $Q'$ and $x$.
This gives the desired outcome of the lemma.

\begin{figure}
    \centering
    \begin{tikzpicture}[scale=1.25]
    \pgfdeclarelayer{background}
    \pgfdeclarelayer{foreground}
    \pgfsetlayers{background,main,foreground}
    \begin{pgfonlayer}{background}
        \pgftext{\includegraphics[width=6cm]{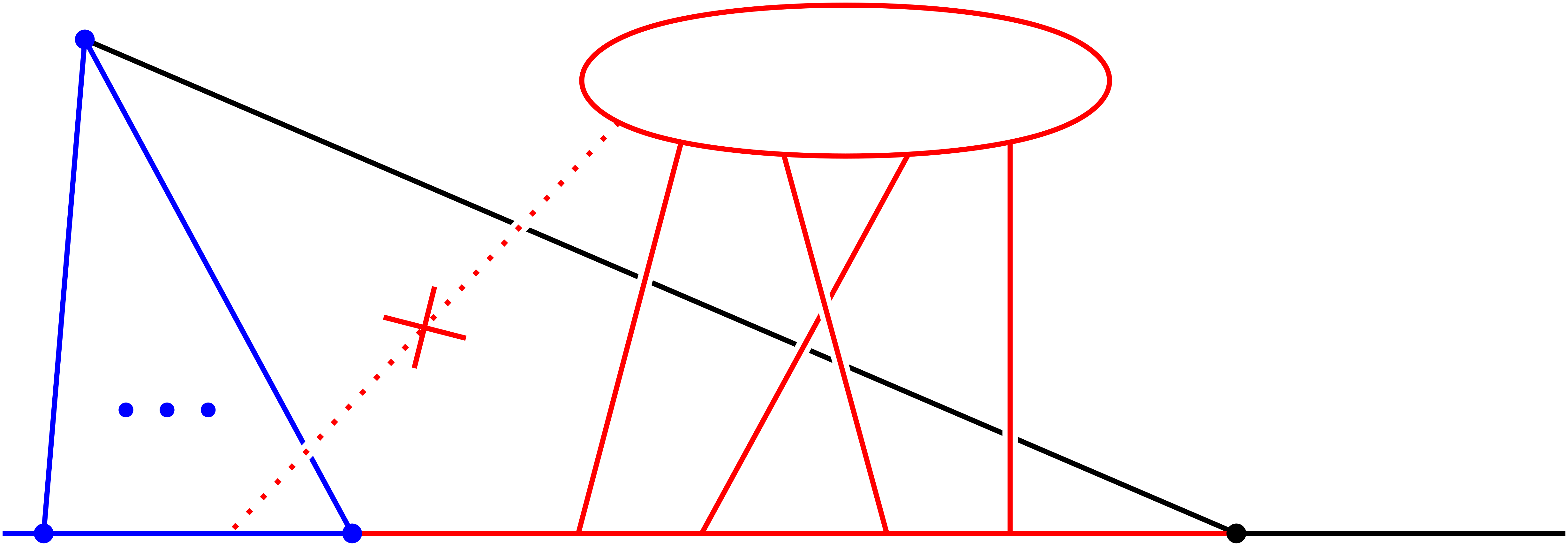}} at (C.center);
    \end{pgfonlayer}{background}
    \begin{pgfonlayer}{main}
        \node[]() at (-2.45,1.03) {$a_1$} ;
        \node[]() at (-3.13,-0.96) {$x$} ;
        \node[]() at (-2.6,-1.18) {$v_1^1$} ;
        \node[]() at (-1.43,-1.2) {$v_j^1$} ;
        \node[]() at (1.88,-1.2) {$v_{j+1}^1$} ;
        \node[]() at (0.38,0.72) {$\mathcolor{red}{A''}$} ;
    \end{pgfonlayer}{main}
    \begin{pgfonlayer}{foreground}
    \end{pgfonlayer}{foreground}
\end{tikzpicture}
\hspace{0.2cm}
\begin{tikzpicture}[scale=1.25]
    \pgfdeclarelayer{background}
    \pgfdeclarelayer{foreground}
    \pgfsetlayers{background,main,foreground}
    \begin{pgfonlayer}{background}
        \pgftext{\includegraphics[width=6cm]{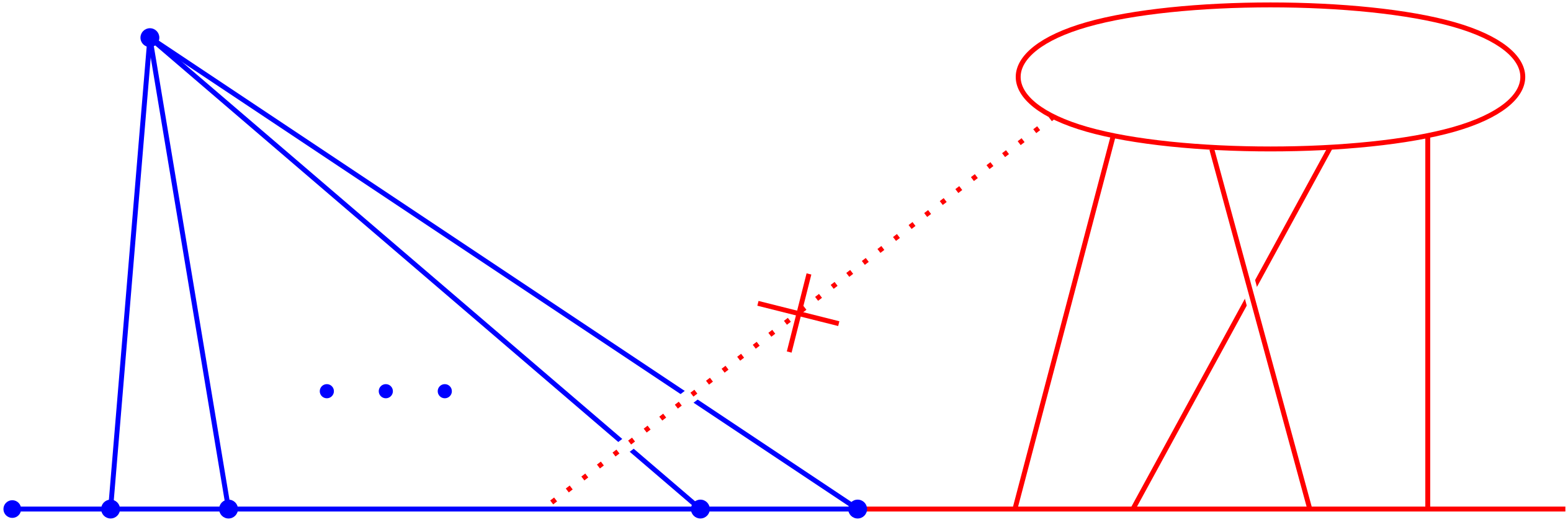}} at (C.center);
    \end{pgfonlayer}{background}
    \begin{pgfonlayer}{main}
        \node[]() at (-2.45,1.03) {$a_1$} ;
        \node[]() at (-3.13,-0.96) {$x$} ;
        \node[]() at (-2.6,-1.18) {$v_1^1$} ;
        \node[]() at (0.48,-1.2) {$v_{\deg{(a_1)}}^1$} ;
        \node[]() at (1.9,0.72) {$\mathcolor{red}{A''}$} ;
    \end{pgfonlayer}{main}
    \begin{pgfonlayer}{foreground}
    \end{pgfonlayer}{foreground}
\end{tikzpicture}
    \caption{The two cases in the proof of \cref{cleanoneside}; \textbf{Cases 1} (Left): There are many vertices in $A'$ that have no neighbor before $v_j^1$ but many of them have a neighbor before $v_{j+1}^1$.
    In this case, we discard $v^1_{j+1}$ and all vertices after it. 
    \textbf{Case 2} (Right): There are many vertices in $A'$ that have no neighbor before $v^1_{\deg(a_1)}$.}
    \label{fig:cleancase1}
\end{figure}
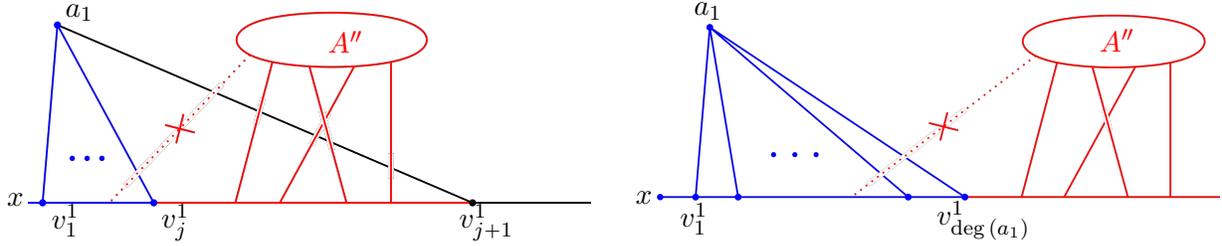

\medskip

\textbf{Case 2:} Assume that for each $j\in [2d-3]$, $\mathsf{gap}(v_j^1,v_{j+1}^1)< (2d-2)(2d-1)^{j-1}\rho(k-1,d)$ holds.
Then the number of vertices in $A'\setminus\{a_1\}$ that have a neighbor in $\mathring{v}_1^1P\mathring{v}_{\deg(a_1)}^1$ is
\begin{align*}
    (A'\setminus\{a_1\}) \cap N(\mathring{v}_1^1P\mathring{v}_{\deg(a_1)}^1) \leq ~& \sum_{j\in [\deg(a_1)-1]} \mathsf{gap}(v_j^1,v_{j+1}^1)\\
    < ~& \sum_{j\in [\deg(a_1)-1]}(2d-2)(2d-1)^{j-1}\rho(k-1,d)\\
    = ~& ((2d-1)^{\deg(a_1)}-1)\rho(k-1,d)\\
    \leq ~& ((2d-1)^{2d-2}-1)\rho(k-1,d).
\end{align*}
Let $A'' \coloneqq (A'\setminus \{a\})\setminus N(P\mathring{v}_{\deg(a_1)}^1)$ be the set of all members of $A'$ that have no neighbor in $P\mathring{v}_{\deg(a_1)}^1$.
Then $|A''|\geq \rho(k-1,d)$, so we can use the induction hypothesis on $A''$ and $\mathring{v}^1_{\deg(a_1)}P$ to obtain $B'\subseteq A''$ and a subpath $Q'$ of $P$.
Then let $B=B'\cup \{a_1\}$ and let $Q$ be the minimal subpath of $P$ containing $Q'$ and $x$.
This gives the desired outcome of the lemma, which completes the proof.
\end{proof}

By using \Cref{cleanoneside} twice and applying the Theorem of Erd\H{o}s and Szekeres on monotone subsequences \cite{ErdosSzekeres}, we can show that a large shuffled rope ladder contains a rope ladder as an induced subgraph.

\begin{theorem}[Erd\H{o}s and Szekeres \cite{ErdosSzekeres}]\label{ErdosSzekeres}
    Let $a$ and $b$ be positive integers, $n \geq (a-1)(b-1)+1$, and $x_1,x_2,\cdots,x_n\in \mathbb{R}$ be a sequence of real numbers.
    Then we can either find a monotonically increasing subsequence of length $a$ or a monotonically decreasing subsequence of length $b$.  
\end{theorem}

\begin{lemma}\label{lem:cleanshuffledropeladder}
    Let $k$ and $d\geq 2$ be positive integers.
    Let $\nu(k,d)\coloneqq \rho(\rho((k-1)^2+1,d),d)$.
    If a $\nu(k,d)$-shuffled rope ladder is $K_{1,d}$-free, then it has a $k$-rope ladder as an induced subgraph.
\end{lemma}
\begin{proof}
    Let $P_1$ and $P_2$ be the rail paths of a $\nu(k,d)$-shuffled rope ladder, and let $\Phi^1,\Phi^2,\cdots,\Phi^{\nu(k,d)}$ be its rung paths.
    Let $\phi_1^i\in N(P_1)$ be one endpoint of $\Phi^i$ and let $\phi_2^i\in N(P_2)$ be its other endpoint.
    By applying \Cref{cleanoneside} on $P_1$ and $\{\phi_1^1,\phi_1^2,\cdots,\phi_1^{\nu(k,d)} \}$, we may find a subpath $Q_1$ of $P_1$ and a subset $B=\{b_1,b_2,\cdots,b_{\rho((k-1)^2+1,d)}\}\subseteq [\nu(k,d)]$ such that the $\{N(\phi_1^{b_i})\cap V(Q_1)\}_{i\in [\rho((k-1)^2+1,d)]}$ appear in order along $Q_1$.
    Now apply \Cref{cleanoneside} again on $P_2$ and $\{\phi_2^{b_1},\phi_2^{b_2}, \cdots,\phi_2^{b_{\rho((k-1)^2+1,d)}}\}$ to find a subpath $Q_2$ of $P_2$ and subset $B'=\{b'_1,b'_2,\cdots,b'_{(k-1)^2+1}\}\subseteq B$ such that the $\{N(\phi_2^{b'_j})\cap V(Q_2)\}_{j\in [(k-1)^2+1]}$ appear in order along $Q_2$.
    
    By renaming the indices of $\Phi^i$'s in $B'$ to $[(k-1)^2+1]$, we may assume that the $\{N(\phi_1^{i})\cap V(Q_1)\}_{i\in [(k-1)^2+1]}$ appear in order along $Q_1$ and $\{N(\phi_2^{\sigma(i)})\cap V(Q_2)\}_{i\in [(k-1)^2+1]}$ appears in order along $Q_2$, where $\sigma$ is a permutation on $[(k-1)^2+1]$.
    Now we may use \Cref{ErdosSzekeres} to obtain a subset $B''\subseteq [(k-1)^2+1]$ of size $k$ such that $\sigma(b)$ for $b\in B''$ monotonically increases or monotonically decreases.
    By deleting all rung paths other than whose indices are in $B''$, we now find a $k$-rope ladder as an induced subgraph using $Q_1$ and $Q_2$ as a rail paths, and the $\{\Phi_i\}_{i\in B''}$ as rung paths.
\end{proof}

Hence, our goal from now on is to find a $\nu(k,d)$-shuffled rope ladder as an induced subgraph in a graph with sufficiently large tree-independence number.

\section{Finding rail paths for a shuffled rope ladder}\label{sec:DividingPath}
In this section, we show that in a $K_{1,d}$-free graph with large tree-independence number, we can find a non-adjacent induced path $P$ and subgraph $H$, where $H$ is either an induced path or an induced cycle, and any separator between their neighbor has large independence number.
In \Cref{sec:MainTheorem}, these induced subgraphs will be the starting place to find a shuffled rope ladder or a cycle rope ladder.
A core feature of our proof, which allows us to circumvent the absence of an (approximate) induced Menger Theorem for $K_{1,d}$-free graphs, is that throughout our proofs the choice of the graph $H$ will be fixed, but we allow to replace the path $P$ with a different path in certain circumstances.

\begin{theorem}\label{thm:pathandcycle}
    Let $d\geq 2$ and $\eta\geq d$ be positive integers and let $G$ be a $K_{1,d}$-free graph with $\atw(G)\geq 8\eta+8d-14$.
    Then $G$ contains an induced path $P$ and an induced cycle $C$ such that
    \begin{itemize}
        \item $P$ and $C$ are non-adjacent, and
        \item if $S$ is a separator between $N[P]$ and $N[C]$, then $\alpha(S)\geq \eta$.
    \end{itemize}
\end{theorem}
\begin{proof}
By \Cref{thm:alphaduality}, there exists a strong bramble $\mathcal{B}$ with $\alpha$-order at least $2\eta+2d-3$ in $G$.
Using \Cref{lemma:dominatingpath}, we can find a path $Q=v_1v_2\cdots v_m$ such that $N[Q]\cap B\neq \emptyset$ for all $B\in \mathcal{B}$.

First, we divide $Q$ into paths $P$ and $P'$ by removing one vertex in the middle of the path.
For each $i\in [m]$, let $\mathcal{B}^i\coloneqq\{B\in \mathcal{B}\mid B\cap N[v_1Qv_i]\neq \emptyset\}$.
Let $a_i$ be the $\alpha$-order of $\mathcal{B}^i$, and let $A^i$ be a hitting set of $\mathcal{B}^i$ with independence number $a_i$.
We claim that $a_{i}\leq a_{i+1}\leq a_{i}+d-1$ for each $i\in [m-1]$.
The first inequality comes from $\mathcal{B}^i\subseteq \mathcal{B}^{i+1}$, and for the second inequality, observe that $A^i\cup N[v_i]$ is a hitting set for $\mathcal{B}^{i+1}$.
Hence, $a_{i+1}\leq \alpha(A^i\cup N[v_i])\leq a_i+d-1$.
Also, we have $a_1\leq d-1$ since $N[v_1]$ is a hitting set for $\mathcal{B}^1$.
$a_i$ gradually grows from $a_1\leq d-1$ to $a_{m}\geq 2\eta+2d-3$, increasing by at most $d-1$ at each step, so we can choose $j\leq m-2$ such that $\eta \leq a_{j}\leq \eta+d-2$.
Now we let $P\coloneqq P^j$.

Let $\mathcal{B}'\coloneqq \mathcal{B} \setminus \mathcal{B}^{j+1}$, and let $R$ be a minimal subpath of $Q$ such that $N[R]$ is a hitting set of $\mathcal{B}'$.
Then $R$ is non-adjacent with $P$.
Let $a'$ be the $\alpha$-order of $\mathcal{B}'$, and let $A'$ be a hitting set of $\mathcal{B}'$ with independence number $a'$.
Then $A^j\cup A'\cup N[v_{j+1}]$ is a hitting set of $\mathcal{B}$, so we have $a'\geq\eta$.

Next, we construct a cycle $C$ by using the same method as in the proof of Theorem 5.2. from \cite{choi2025excludinginducedwheelminor}.
Let $u$ and $v$ be endpoints of $R$.
As $(d-1)\cdot |V(R)| \geq \alpha(N[R])\geq a' \geq \eta$, we have $|V(R)|\geq  2$, which implies $u\neq v$. 
By the minimality of $R$, there exist $B_u,B_v\in \mathcal{B}'$ such that $N[B_u]\cap V(R)=\{u\}$ and $N[B_v]\cap V(R)=\{v\}$.
As $B_u\cap B_v\neq \emptyset$, we can find an $u$-$v$ path $R'$ whose internal vertices are in $B_u\cup B_v$.
Furthermore, no internal vertex of $R'$ has a neighbor in $R$ other than $u$ and $v$.
Hence, $R\cup R'$ gives an induced cycle $C$.

Now we show that $P$ and $C$ satisfy the desired conditions.
First, $V(C)=V(R)\cup V(R')\subseteq V(P')\cup V(B_u)\cup V(B_v)$ implies that $P$ and $C$ are non-adjacent.
To see that the second condition is met, suppose that there is a separator $S$ between $N[P]$ and $N[C]$ with $\alpha(S)<\eta$.
Then $S$ cannot be a hitting set of $\mathcal{B}^j$ nor of $\mathcal{B}'$ so there are $B_1\in \mathcal{B}^j$ and $B_2\in \mathcal{B}'$ such that $B_1\cap S=B_2\cap S=\emptyset$.
Since $N[P]\cap B_1$, $B_1\cap B_2$, and $B_2\cap N[C]$ are nonempty, there exists a path $X$ between $N[P]$ and $N[C]$ whose vertices are contained in $B_1\cup B_2$.
This contradicts $S$ being a separator between $N[P]$ and $N[C]$.
Therefore, we have $\alpha(S)\geq \eta$.
\end{proof}

Note that by using the same method, it is also possible to show that any separator $S$ between $P$ and $R$ also satisfies $\alpha(S)\geq \eta$ in the above proof.
For later use, state this as a separate theorem.

\begin{theorem}\label{thm:twopaths}
    Let $d\geq 2$ and $\eta\geq d$ be a positive integers and let $G$ be a $K_{1,d}$-free graph with $\atw(G)\geq 8\eta+8d-14$.
    Then $G$ contains two non-adjacent induced paths $P_1$ and $P_2$ such that if $S$ is a separator between $N[P_1]$ and $N[P_2]$ in $G$, then $\alpha(S)\geq \eta$.
\end{theorem}

Lastly, we remark that by slightly modifying the proof, we can show that we can change the second condition of $\Cref{thm:pathandcycle}$ to be about a separator between $P$ and $C$, instead of $N[P]$ and $N[C]$.

\begin{theorem}\label{thm:pathandcycle2}
    Let $d\geq 2$ and $\eta\geq d$ be positive integers and let $G$ be a $K_{1,d}$-free graph with $\atw(G)\geq 16d\eta$.
    Then $G$ contains an induced path $P$ and an induced cycle $C$ such that
    \begin{itemize}
        \item $P$ and $C$ are non-adjacent, and
        \item if $S$ is a separator between $P$ and $C$ in $G$, then $\alpha(S)\geq \eta$.
    \end{itemize}
\end{theorem}
\begin{proof}
Observe that $16d\eta\geq 4(2\eta+2(d-1)(2\eta-2)+2d-3)-2$.
By \Cref{thm:alphaduality}, there exists a strong bramble $\mathcal{B}$ with $\alpha$-order at least $2\eta+2(d-1)(2\eta-2)+2d-3$.
Using \Cref{lemma:dominatingpath}, we can find a path $Q=v_1v_2\cdots v_m$ such that $N[Q]\cap B\neq \emptyset$ for all $B\in \mathcal{B}$.

We construct non-adjacent path $P$ and cycle $C$ similarly as in the proof of \Cref{thm:pathandcycle}:
$N[P]$ is a hitting set of $\mathcal{B}^P\subseteq \mathcal{B}$ and $N[C]$ is a hitting set of $\mathcal{B}^C\subseteq \mathcal{B}$ where $\alpha$-order of $\mathcal{B}^P$ and $\mathcal{B}^C$ are at least $\eta+(d-1)(2\eta+2)$. 

Suppose that there is a separator $S$ between $P$ and $C$ with $\alpha(S)<\eta$.
Let $S'=S\cup N[S\cap P]$.
Since $|S\cap P|\leq 2\eta-2$, we have $\alpha(S')\leq \alpha(S)+\alpha(N[S\cap P])<\eta+(d-1)(2\eta-2)$.
Hence $S'$ cannot be a hitting set of $\mathcal{B}^j$.
Therefore, there is $B_1\in \mathcal{B}^j$ such that $B_1\cap S'=\emptyset$.
Similarly, let $S''=S\cup N[S\cap C]$.
Then $\alpha(S'')<\eta+(d-1)(2\eta-2)$, so $S''$ cannot be a hitting set of $\mathcal{B}'$.
Therefore, there is $B_2\in \mathcal{B}'$ such that $B_2\cap S''=\emptyset$.

Since $N[P]\cap B_1$, $B_1\cap B_2$, and $B_2\cap N[C]$ are nonempty, we can find a path $R$ between $N[P]$ and $N[C]$ whose vertices are contained in $B_1\cup B_2$.
Let $x\in N[P]$ be an endpoint of $R$.
Suppose that $x\notin V(P)$.
As $x\notin N[S\cap P]$, there is $x'\in V(P)\setminus S$ adjacent to $x$.
If $x'$ is not already in $R$, we may extend $R$ by adding $x'$ as a new endpoint.
Therefore, we may assume that $X \cap P\neq \emptyset$.
Similarly, we may assume that $X\cap C\neq \emptyset$.
This contradicts to $S$ being a separator between $P$ and $C$.
Therefore, we have $\alpha(S)\geq \eta$.
\end{proof}

\section{Finding a $k$-$H$-rope ladder}\label{sec:MainTheorem}
In this section, we prove that if we have an induced path $P$ and an induced subgraph $H$ in a graph $G$ such that $P$ and $H$ are non-adjacent and any separator between them has large independence number, then we can find a $k$-$H$-rope ladder as an induced subgraph in $G$.
Note that $H$ is fixed here, but the path $P$ does not \textsl{have} to be the rail path of the final $k$-$H$-rope ladder.

\begin{theorem}\label{thm:somethingpath}
    Let $\ell$ and $d\geq 2$ be positive integers and let $G$ be a $K_{1,d}$-free graph.
    Let $P$ be an induced path of $G$ and let $H$ be a subgraph of $G$ that is not adjacent with $P$.
    If every separator $S$ between $N[P]$ and $N[H]$ satisfies $\alpha(S)\geq \eta(\ell,d)\coloneqq 8(d-1)\ell^{\ell+1}$, then there exists an $\ell$-$H$-rope ladder $L$ as an induced subgraph in $G$, where $H$ is a rail subgraph of $L$.
\end{theorem}

\medskip
If $H$ is also a path, the result of above theorem forms an $\ell$-shuffled rope ladder.
Hence, we may deduce \Cref{thm:themainresult} by combining \Cref{thm:somethingpath} with \Cref{lem:cleanshuffledropeladder} and \Cref{thm:twopaths} as follows.

\begin{proof}[Proof of \Cref{thm:themainresult}]
Let $\tau(k,d)\coloneqq 8 \eta(\nu(k,d),d)+8d-14$.
Suppose that a $K_{1,d}$-free graph $G$ has $\atw(G)\geq \tau(k,d)$.
Then \Cref{thm:twopaths} implies the existence of two non-adjacent induced paths $P_1$ and $P_2$ such that every separator $S$ between $N[P_1]$ and $N[P_2]$ satisfies $\alpha(S)\geq \eta(\nu(k,d),d)$.
Next, by \Cref{thm:somethingpath}, there exists a $\nu(k,d)$-shuffled rope ladder as an induced subgraph of $G$.
Lastly, we can use \Cref{lem:cleanshuffledropeladder} to find a $k$-rope ladder as an induced subgraph of $G$, which has the $k$-skinny ladder as an induced minor.
\end{proof}

Before we dive into the proof of \Cref{thm:somethingpath}, we briefly sketch an outline of our approach. 
We wish to inductively construct an induced subgraph between $H$ and $P$, each time adding several disjoint paths to the current subgraph.
For each newly added path, one endpoint of it will be adjacent to $H$, the other endpoint is adjacent to one of the paths constructed in the previous step, and internal vertices are non-adjacent to other paths.
If we manage to find a certain number of such paths, we will have obtained our desired structure by using these paths as rail paths.
Otherwise, we find ourselves in a situation that permits the use of the condition that a separator between $N[H]$ and $N[P]$ has large independent number to move on to the next step.
If this process repeats $\ell+1$ times, we will have found our desired structure by selecting one of the paths from each step and combining them into the second rail path.
To guarantee that the paths we find are indeed induced and non-adjacent to each other in the way we want, when we find non-adjacent paths, we greedily choose shortest paths, and manage the endpoints and second-endpoints of these paths separately.

\begin{proof}[Proof of \Cref{thm:somethingpath}]
We may assume that $P$ and $H$ are in the same component of $G$, since otherwise $\emptyset$ is a separator between $N[P]$ and $N[H]$ with independence number $0$.
Furthermore, we may assume that $G$ is connected.
Also, if $\ell=1$, we may choose a shortest path from $H$ and $P$ to be a rung path, so we may further assume that $\ell\geq 2$.
Let $G^0=G-(N[H]\cup N[P])$, and let $N^0\subseteq N[H]$ be the set of vertices reachable from $N[P]$ through $G^0$.
Then we have $\alpha(N^0)\geq \eta(\ell,d)$ since $N^0$ is a separator between $N[P]$ and $N[H]$.
As the first step of the construction is slightly different from the subsequent steps, and for a better overview of the proof, we describe the first step separately.
\medskip

\paragraph{First step.}
If there are $\ell$ pairwise non-adjacent paths between $N[H]$ and $N[P]$, we may choose them as rung paths and choose $P$ as a rail path to find the desired induced subgraph. 
Thus suppose that there are at most $\ell-1$ such paths.
We greedily choose such paths one by one, selecting a shortest possible path that is non-adjacent to the already chosen paths, until such a selection is no longer possible.
Let $Q_1^1,\cdots,Q_{m_1}^1$ be the chosen paths, and let $\mathcal{Q}^1=\bigcup_{j\in [m_1]}V(Q_j^1)$.
Then we have $m_1<\ell$, and $N[\mathcal{Q}^1]$ separates $N[H]$ and $N[P]$.
For each $j\in[m_1]$ let $v_j^1\in N[H]$ be one endpoint of $Q_j^1$, and let $w_j^1$ be the other endpoint.
If the length of $Q_j^1$ is not $0$, let $\widehat{v}_j^1$ and $\widehat{w}_j^1$ be the second-endpoints of $Q_j^1$, adjacent to $v_j^1$ and $w_j^1$, respectively.
Note that we have $w_j^1\not\in N[H]$ unless $w_j^1=v_j^1$.
Let $S^1$ be the set of endpoints and second-endpoints of the $Q^{1}_j$'s.
We do not want the endpoints of distinct paths from our selection to be adjacent, so we will handle these vertices in each path separately.
In particular, we forbid the paths found in later steps from being adjacent to the vertices in $S^1$.
Let $\overline{\mathcal{Q}}_j^1\coloneqq N[Q_j^1]\setminus N[S^1]$ and let $\overline{\mathcal{Q}}^1\coloneqq \bigcup_{j\in [m_1]}\overline{\mathcal{Q}}_j^1$.
Let $N^1\subseteq N^0\setminus N[S^1]$ be the set of all vertices in $N^0\setminus N[S^1]$ that are reachable from $\overline{\mathcal{Q}}^1$ through $G^1\coloneqq G^0-\bigcup_{j\in [m_1]}N[Q_j^1]$.

\begin{claim}\label{claimc}
    $N^1\cup N[S^1]$ is a separator between $N[H]$ and $N[P]$.
\end{claim}
Suppose the claim is false and let $\widetilde{Q}$ be a path from $N[H]$ to $N[P]$ that does not meet $N^1\cup N[S^1]$.
Let $x\in N[H]$ be one endpoint of $\widetilde{Q}$.
By the maximality of $m_1$, $\widetilde{Q}$ should be adjacent to some $Q_j^1$.
However, $\widetilde{Q}$ does not contain any vertices from $N[S^1]$, so $\widetilde{Q}$ must meet $\overline{\mathcal{Q}}^1$.
Then $x$ is reachable from $\overline{\mathcal{Q}}^1$ through $G^1$, which contradicts $x\not\in N^1$.
This proves the claim.
\medskip

By our assumption on the independence number of separators between $N[H]$ and $N[P]$, we have $\alpha(N^1\cup N[S^1])\geq \eta(\ell,d)$, which gives
\[\alpha(N^1)\geq \eta(\ell,d)-\alpha(N[S^1])\geq \eta(\ell,d)-4(d-1)\ell>0.\]
In particular, we have $N^1\neq \emptyset$.

\begin{claim}\label{claime}
    $\overline{\mathcal{Q}}^1\cap N[H]=\overline{\mathcal{Q}}^1\cap N[P]=\emptyset$.
\end{claim}

Suppose there exists $x\in \overline{\mathcal{Q}}^1\cap N[H]$.
Then there is some $j$ such that $x\in N[Q_j^1]-N[S^1]$.
Choose the minimum such $j$.
Let $y$ be a vertex in $N(x)\cap Q_j^1$ that is closest to $w_j^1$.
Since $x\not\in S^1$, we have $y\not\in\{ v_j^1,\widehat{v}_j^1\}$.
Then $w_j^1 Q_j^1 y x$ is a path between $N[H]$ and $N[P]$, and by the minimality of $j$, it is not adjacent with $Q_1^1,\cdots,Q_{j-1}^1$.
Also $w_j^1Q_j^1wx$ has shorter length than $Q_j^1$, but this is a contradiction, as we chose $Q_j^1$ to be a shortest possible path.
Hence $\overline{\mathcal{Q}}^1\cap N[H]=\emptyset$.
Similarly, we can show that $\overline{\mathcal{Q}}^1\cap N[P]=\emptyset$.
This proves the claim.
\medskip

We now continue with the remaining iterative procedure for our construction.

\paragraph{The $i$-th step ($2\leq i\leq \ell+1$).}
Let us assume that in the $(i-1)$-th step, we have defined the following eight objects: (Consider $\mathcal{Q}^0=P$, $\overline{\mathcal{Q}}^{0}=N[P]$ and $S^{0}=\emptyset$ as above.)
\begin{enumerate}[i)]
    \item $Q_1^{i-1},\cdots,Q_{m_{i-1}}^{i-1}$ are pairwise non-adjacent paths between $N[H]$ and $N[\overline{\mathcal{Q}}^{i-2}]$ through $G^{i-2}$, found by choosing a shortest path that is non-adjacent with any of the previously chosen paths each time.
    \item $\mathcal{Q}^{i-1}=\bigcup_{j\in [m_{i-1}]}V(Q_j^{i-1})$
    \item $v_{j}^{i-1}\in N[H]$ and $w_{j}^{i-1}$ are the endpoints of $Q_j^{i-1}$.
    If the length of $Q_j^{i-1}$ is not $0$, $\widehat{v}_{j}^{i-1}$ and $\widehat{w}_{j}^{i-1}$ are the second-endpoints of $Q_j^{i-1}$ adjacent to $v_{j}^{i-1}$ and $w_{j}^{i-1}$, respectively.
    \item $S^{i-1}=S^{i-2}\cup \bigcup_{j\in [m_{i-1}]}\{v_j^{i-1},w_{j}^{i-1},\widehat{v}_j^{i-1},\widehat{w}_j^{i-1}\}$ is the set of endpoints and second-endpoints of the $Q^{i-1}_j$ (note that $\widehat{v}_j^{i-1},\widehat{w}_j^{i-1}$ might not always exist) together with all such vertices chosen in the previous steps.
    \item $\overline{\mathcal{Q}}_j^{i-1}=N[Q_j^{i-1}]\setminus[S^{i-1}]$
    \item $\overline{\mathcal{Q}}^{i-1}=\bigcup_{j\in [m_{i-1}]}\overline{Q}_j^{i-1}$
    \item $G^{i-1}=G^{i-2}-\bigcup_{j\in [m_{i-1}]} N[Q^{i-1}_j]$
    \item $N^{i-1}\subseteq N^{i-2}\setminus N[S^{i-1}]$ is the set of vertices that are reachable from $\overline{\mathcal{Q}}^{i-1}$ through $G^{i-1}$.
\end{enumerate}

We further assume that these objects satisfy the following five properties:
\begin{enumerate}[a)]
    \item $m_{i-1}< \ell^{i-1}$ \label{item_a}
    \item $\overline{\mathcal{Q}}^{i-1}\cup S^{i-1}$ is a separator between $N[H]$ and $\overline{\mathcal{Q}}^{i-2}$. \label{item_b}
    \item $N^{i-1}\cup N[S^{i-1}]$ is a separator between $N[H]$ and $N[P]$. \label{item_c}
    \item $\alpha(N^{i-1})\geq \eta(\ell,d)-4(d-1)\sum_{j\in [i-1]}\ell^j>0$.
    In particular, $N^{i-1}\neq \emptyset$. \label{item_d}
    \item $\overline{\mathcal{Q}}^{i-1}\cap N[H]=\overline{\mathcal{Q}}^{i-1}\cap \overline{\mathcal{Q}}^{i-2}=\emptyset$. \label{item_e}
\end{enumerate}

If there are at least $\ell^i>\ell\cdot m_{i-1}$ non-adjacent paths between $N^{i-1}$ and $\overline{\mathcal{Q}}^{i-1}$, then there is some $j\in [m_{i-1}]$ such that there are at least $\ell$ non-adjacent paths between $N^{i-1}$ and $\overline{\mathcal{Q}}_j^{i-1}$.
In this case, we may choose those paths as rung paths and choose $Q_j^{i-1}$ as a rail path to find the desired induced subgraph.
Thus suppose there are at most $\ell^{i}-1$ such paths.
As before, we choose such paths between $N^{i-1}$ and $\overline{\mathcal{Q}}^{i-1}$ one by one, selecting a shortest possible path that is non-adjacent to any of the previously chosen paths, until all options are depleted.
Let $Q_1^{i},\cdots,Q_{m_i}^i$ be the chosen paths.
Note that since $N^{i-1}\neq \emptyset$, we have $m_i\geq 1$.
Then we define $\mathcal{Q}^i,S^i,\overline{\mathcal{Q}}_{j}^{i},\overline{\mathcal{Q}}^{i},G^i,$ and $N^i$ analogously to before.
Now we have to show that properties \ref{item_a} - \ref{item_e} hold for $i$.

\paragraph{\ref{item_a}}
We have $m_i<\ell^i$ from the assumption above.

\paragraph{\ref{item_b}}
From the definition of $m_i$, it follows directly that $\overline{\mathcal{Q}}^i\cup S^i$ is a separator between $N[H]$ and $N[\overline{\mathcal{Q}}^{i-2}]$.

\paragraph{\ref{item_c}}
The proof is carried out analogously to the proof of \Cref{claimc}; by using property \ref{item_c} from prior steps.
Suppose that $N^i\cup N[S^i]$ is not a separator between $N[H]$ and $N[P]$.
Then there exists a path $\widetilde{Q}^0$ from $N[H]$ to $N[P]$ avoiding $N^{i}\cup N[S^{i}]$.
We inductively find a subpath of $\widetilde{Q}^0$ between $N[H]$ and $\overline{\mathcal{Q}}^0$, whose internal vertices are in $G^i$.
As $\overline{\mathcal{Q}}^1\cup S^1$ is a separator between $N[H]$ and $N[P]$, we have $V(\widetilde{Q}^0)\cap \overline{\mathcal{Q}}^1\neq \emptyset$.
Thus we can find a subpath of $\widetilde{Q}^0$ between $N[H]$ and $\overline{\mathcal{Q}}^1$ that does not meet with $\overline{\mathcal{Q}}^1$ again, say $\widetilde{Q}^1$.
Note that the internal vertices of $\widetilde{Q}^1$ are contained in $G^1$.
We can repeat this process (considering $N[P]=\overline{\mathcal{Q}}^0$) to find a path $\widetilde{Q}^i$ between $N[H]$ and $\overline{\mathcal{Q}}^i$, whose internal vertices are in $G^i$.
Then one endpoint of $\widetilde{Q}^i$ is in $N[H]$, and by the definition of $N^i$, that endpoint is in $N^i$.
But this contradicts our assumption that $\widetilde{Q}^0$ avoids $N^i$.
Therefore, $N^i\cup N[S^i]$ is a separator between $N[H]$ and $N[P]$.

\paragraph{\ref{item_d}} As $N^i\cup N[S^i]$ is a separator between $H$ and $P$, we have $\alpha(N^{i}\cup N[S^i])\geq \eta(\ell,d)$.
Hence
\begin{align*}
    \alpha(N^{i})&\geq \eta(\ell,d)-\alpha(N[S^{i}])\\
    &\geq \eta(\ell,d)-(d-1)|S^{i}|\\
    &\geq \eta(\ell,d)-4(d-1)\sum_{j\in [i]}m_{j}\\
    &> \eta(\ell,d)-4(d-1)\sum_{j\in [i]}\ell^j\\
    &>0.
\end{align*}

\paragraph{\ref{item_e}}
This can be proved along the same arguments as \Cref{claime}.
\medskip

Hence, we constructed eight objects satisfying properties \ref{item_a}-\ref{item_e}.
This completes the $i$-th step.

\paragraph{After $\ell+1$ steps.}
Suppose that the above construction process did not terminate until reaching the $(\ell+1)$-th step.
Note that as $\eta(\ell,d)=8(d-1)\ell^{\ell+1}\geq \sum_{j\in [\ell+1]} 4(d-1)\ell^j $, our choices of numbers do actually allow us to repeat the above construction process for $\ell+1$ times.
Now for each $i\in [\ell+1]$, we find the $i$-th rung path as a subpath of $Q_{j}^i$ for some $j$, and iteratively construct the second rail path by attaching the remaining parts of each $Q_{j}^i$ as follows.

Choose $\ell+1$ paths $R^1,\cdots,R^{\ell+1}$ as follows:
First, let $R^{\ell+1}\coloneqq Q_1^{\ell+1}$.
After we choose $Q_{t_{i+1}}^{i+1}$ as $R^{i+1}$, choose $R^i$ to be $Q_{t_i}^i$ such that one endpoint of $R^{i+1}$ is in the neighbor of $Q_{t_i}^i$, i.e. $w_{t_{i+1}}^{i+1}\in \overline{\mathcal{Q}}_{t_i}^i$.
Let $\phi_1^i\coloneqq v_{t_i}^i$ be an endpoint of $R^i$ in $N[H]$, and let $w^i\coloneqq w_{t_i}^i$ be the other endpoint.
Note that for each $i\in [\ell]$, $R^i$ has length at least $4$, since $w^{i}$ is adjacent to a vertex in $R^i$ that is not an endpoint or second endpoint of $R^i$.

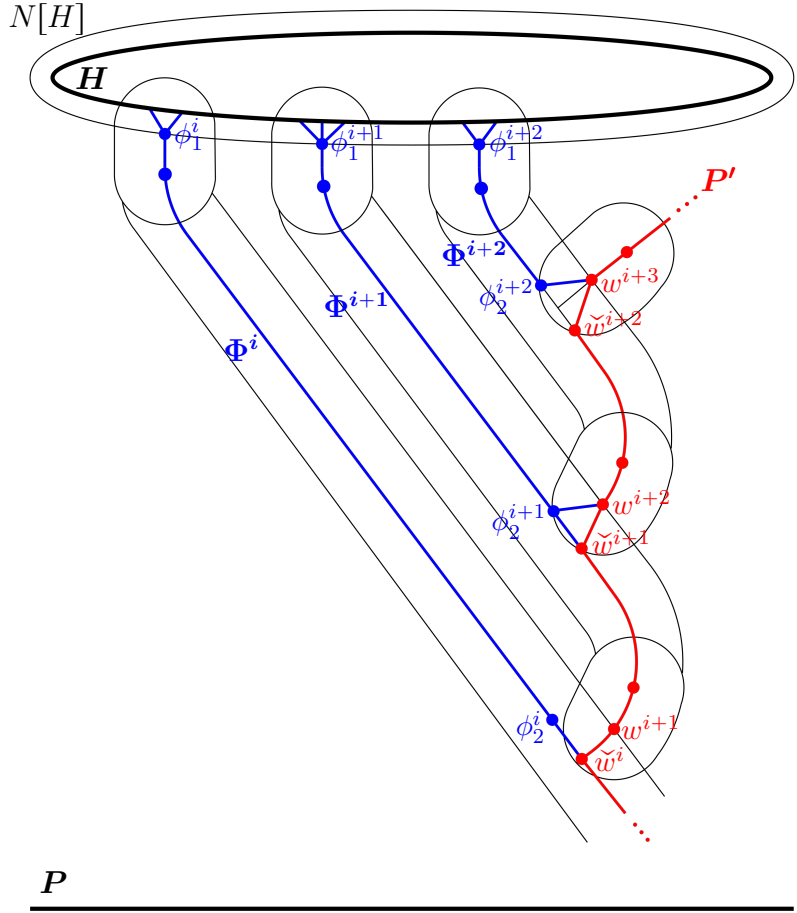
\begin{figure}[H]
    \centering
    \begin{tikzpicture}[scale=1.7]
    \pgfdeclarelayer{background}
    \pgfdeclarelayer{foreground}
    \pgfsetlayers{background,main,foreground}
    \begin{pgfonlayer}{background}
        \pgftext{\includegraphics[width=6cm]{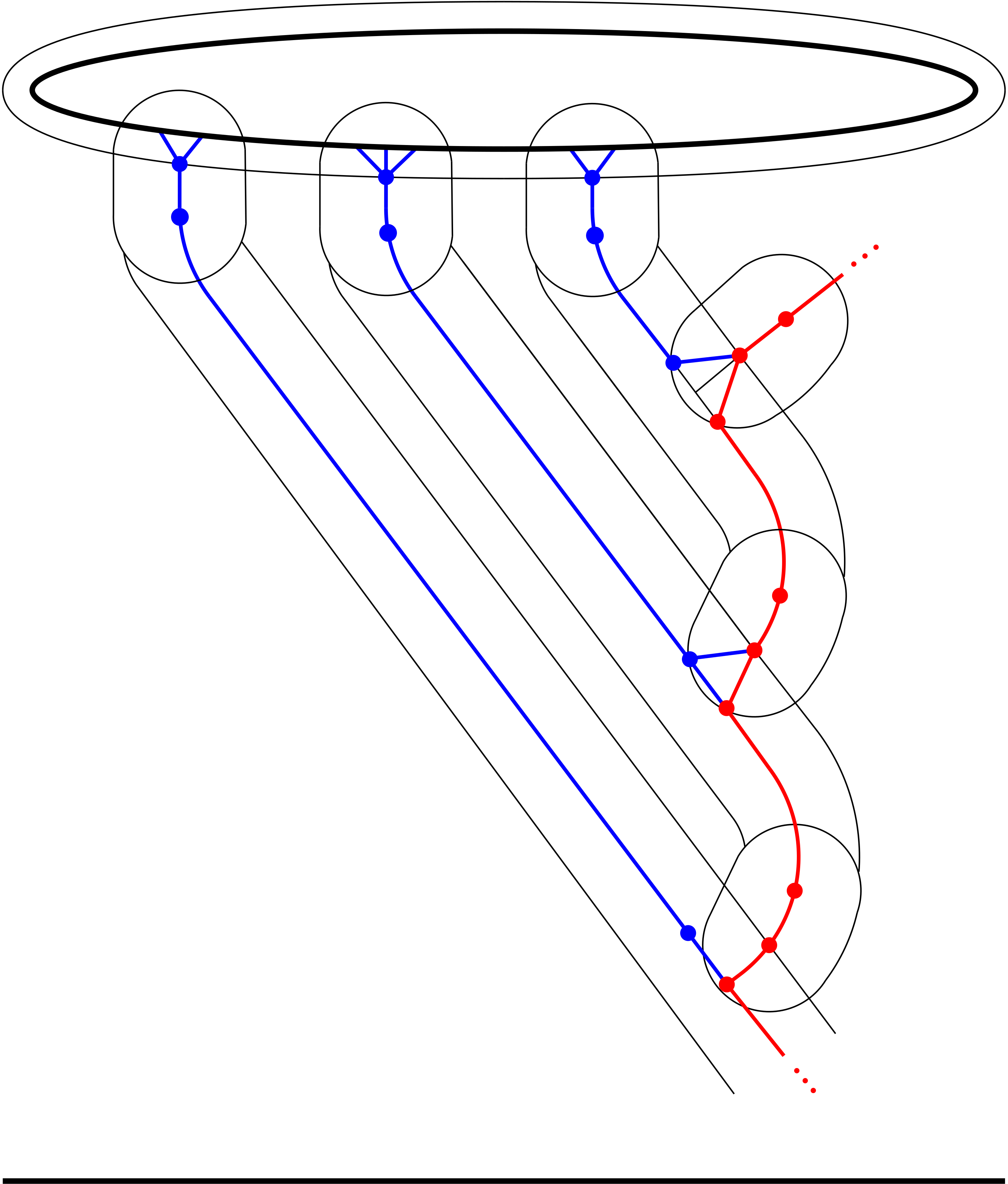}} at (C.center);
    \end{pgfonlayer}{background}
    \begin{pgfonlayer}{main}
        \node[]() at (-2.8,-3.3) {\large $\bm{P}$ };
        \node[]() at (-2.5,2.99) {\large $\bm{H}$ };
        \node[]() at (-2.85,3.45) {\large $N[H]$ };
        \node[]() at (-1.74,2.55) {$\color{blue}{\phi_1^i}$ };
        \node[]() at (-0.42,2.485) {$\color{blue}{\phi_1^{i+1}}$ };
        \node[]() at (0.82,2.485) {$\color{blue}{\phi_1^{i+2}}$ };
        \node[]() at (0.93,-2.09) {$\color{blue}{\phi_2^i}$ };
        \node[]() at (0.84,-0.48) {$\color{blue}{\phi_2^{i+1}}$ };
        \node[]() at (0.74,1.28) {$\color{blue}{\phi_2^{i+2}}$ };
        \node[]() at (1.55,-2.34) {$\color{red}{\widecheck{w}^i}$ };
        \node[]() at (1.87,-2.09) {$\color{red}{w^{i+1}}$ };
        \node[]() at (1.64,-0.64) {$\color{red}{\widecheck{w}^{i+1}}$ };
        \node[]() at (1.8,-0.33) {$\color{red}{w^{i+2}}$ };
        \node[]() at (1.58,1.07) {$\color{red}{\widecheck{w}^{i+2}}$ };
        \node[]() at (1.71,1.41) {$\color{red}{w^{i+3}}$ };
        \node[]() at (-1.33,0.88) {\large $\color{blue}{\bm{\Phi^i}}$ };
        \node[]() at (-0.43,1.22) {\large $\color{blue}{\bm{\Phi^{i+1}}}$ };
        \node[]() at (0.49,1.62) {\large $\color{blue}{\bm{\Phi^{i+2}}}$ };
        \node[]() at (2.4,2.2) {\large $\color{red}{\bm{P'}}$ };
    \end{pgfonlayer}{main}
    \begin{pgfonlayer}{foreground}
    \end{pgfonlayer}{foreground}
\end{tikzpicture}
    \caption{After $\ell+1$ steps, we can find a rail path $P'$ (red) and rung paths $\Phi^i$'s (blue).}
    \label{fig:thelaststep}
\end{figure}

For $i\in [\ell]$, among the vertices in $N(w^{i+1})\cap R^i$, let $\widecheck{w}^i$ be the vertex farthest from $H$ (in $R^i$).
If $|N(w^{i+1})\cap R^i|=1$, let $\phi^i_2$ be the neighbor of $\widecheck{w}^i$ in the path $R^i$ closer to $H$ (in $R^i$).
Otherwise, let $\phi^i_2$ be the vertex in $N(w^{i+1})\cap R^i$ that is nearest to $H$ (in $R^i$).
Let $P'\coloneqq \widecheck{w}^1w^2R^2\widecheck{w}^2w^3R^3\widecheck{w}^3\cdots w^{\ell}R^{\ell}\widecheck{w}^{\ell}w^{\ell+1}$.
Since $w^{i}\in S^{i}$ and $\widecheck{w}^i\in \overline {Q}_{t_i}^{i}$, we have $w^i\neq \widecheck{w}^i$.
In particular, this implies that $P'$ is an induced path.
Let $\Phi^i\coloneqq \phi^i_1R^i\phi^i_2$ for $i\in [\ell]$.
Then we find a $k$-$H$-rope ladder as an induced subgraph, whose rung paths are $\Phi^1,\Phi^2,\cdots,\Phi^\ell$ and whose rail path is $P'$, which completes the proof.
\end{proof}

\section{Substructures in ladder-like graphs}\label{sec:Substructures}

In this section, we prove that our techniques allow us to deduce or generalize several known results concerning the tree-independence number of $K_{1,d}$-free.
This is achieved mostly by analyzing substructures of $k$-$H$-rope ladders.

\subsection{Excluding and packing non-adjacent induced cycles}
A first, particularly important substructure of -- specifically -- $k$-skinny ladders are large families of pairwise non-adjacent long induced cycles.
Recently, Ahn, Gollin, Huynh, and Kwon \cite{ahn2024coarseerdhosposatheorem} proved a coarse version of the celebrated Erd\H{o}s-P{\'o}sa theorem as follows.

\begin{theorem}[\cite{ahn2024coarseerdhosposatheorem}]\label{thm:coarseerdosposa}
    There exists a function $f(k)=\mathcal{O}(k\log k)$ such that for every positive integer $k$, every graph $G$ contains either the disjoint union of $k$ cycles as an induced subgraph, or a set $X$ of vertices with $|X|\leq f(k)$ such that $G-N[X]$ is a forest.
\end{theorem}

When restricted to $K_{1,d}$-free graphs, \cref{thm:coarseerdosposa} implies that every graph either contains the disjoint union of $k$ cycles as an induced subgraph, or there exists a set of small independence number whose deletion removes all cycles in the graph.
This directly implies that $K_{1,d}$-free graphs without induced subgraphs isomorphic to the disjoint union of $k$ cycles have bounded $\alpha$-treewidth.

Indeed, this set of authors proved a version of \Cref{thm:coarseerdosposa} regarding a disjoint union of \textsl{long} cycles.
Note that in this version the cycles do not need to be induced cycles, they must be pairwise non-adjacent however.

\begin{theorem}[\cite{ahn2025coarseerdhosposatheorem}]\label{thm:longcoarseerdosposa}
    There exists a function $f(k,\ell)\in \mathcal{O}(\ell k \log k)$ such that for all positive integers $k$ and $\ell\geq 3$, every graph $G$ contains either $k$ pairwise non-adjacent cycles each of length at least $\ell$, or set $X$ of vertices with $|X|\leq f(k,\ell)$ such that, all cycles of $G-N[X]$ have length at most $\ell-1$.
\end{theorem}

Notice that, due to the Grid Theorem of Robertson and Seymour, \cref{thm:longcoarseerdosposa} still implies that any $K_{1,d}$-free graph that does not contain $k$ pairwise non-adjacent cycles each of length at least $\ell$ must have bounded tree-independence number.

As a consequence of \Cref{thm:themainresult}, we obtain a similar Erd\H{o}s-P\'{o}sa-type theorem for long induced cycles in $K_{1,d}$-free graphs.
In fact, we prove a stronger theorem as follows:

Let $H$ be a graph.
We say that $H$ has the \emph{independence Erd\H{o}s-P{\'o}sa property} in $K_{1,d}$-free graphs if there exists a function $f\colon \mathbb{N}^2\to\mathbb{N}$ such that for every $K_{1,d}$-free graph $G$ and every positive integer $k$, either $G$ contains $k$ pairwise non-adjacent induced subgraphs, each containing $H$ as an induced minor, or there exists a set $X\subseteq V(G)$ such that $\alpha(X)\leq f(d,k)$ and $G-X$ does not contain $H$ as an induced minor.

We prove that for every $k$, every connected induced subgraph of the $k$-skinny ladder has the independence Erd\H{o}s-P\'{o}sa property in $K_{1,d}$-free graphs.
\smallskip

Let $G$ be a graph.
Notice that for every set $X\subseteq V(G)$ and every independent set $X'\subseteq X$ with $|X'|=\alpha(X)$ we have $X\subseteq N[X']$.
Moreover, in $K_{1,d}$-free graphs, for any vertex set $X$, $\alpha(N[X])\leq (d-1)|X|$.
Thus in $K_{1,d}$-free graphs, removing the closed neighbor of a vertex set of bounded size is equivalent to removing a vertex set of bounded independence number.
Therefore, \Cref{thm:erdosposaforskinnyladder} as stated below indeed generalizes \Cref{thm:coarseerdosposa} in $K_{1,d}$-free graph.

\begin{theorem}\label{thm:erdosposaforskinnyladder}
    Let $\ell$ be a positive integer and $H$ be a connected induced subgraph of the $\ell$-skinny ladder.
    Then $H$ has the independence Erd\H{o}s-P{\'o}sa property in $K_{1,d}$-free graphs.
\end{theorem}

\begin{proof}
    Let $\tau$ be the function given in \Cref{thm:themainresult}.

    If $\atw(G)\geq \tau(k(\ell+1),d)$, then by $\Cref{thm:themainresult}$, $G$ contains the $k(\ell+1)$-skinny ladder $L$ as an induced minor.
    Notice that $L$ contains $k$ pairwise non-adjacent copies of the $\ell$-skinny ladder as induced subgraphs.
    Hence, in this case, $G$ has $k$ pairwise non-adjacent induced subgraphs, each containing $H$ as an induced minor.
    
    Therefore, we may assume that $\atw(G)\leq \tau(k(\ell+1),d)$.
    We now claim that for each $i\geq 0$ and graph every $G$ with $\atw(G)\leq \tau(k(\ell+1),d)$, either $G$ contains $i$ pairwise non-adjacent induced subgraphs, each with $H$ as an induced minor, or there is a set $X_i$ of vertices with $\alpha(X_i)\leq i\tau(k(\ell+1),d)$ such that, $G-X_i$ does not contain $H$ as an induced minor.

    We proceed by induction on $i$.
    If $i=0$, the above statement is trivially true.
    Hence we may assume $i\geq 1$ and that above statement is true for all $j\in [0,i-1]$.
    Let $(T,\beta)$ be a tree decomposition of $G$ whose independence number at most $\tau(k(\ell+1),d)$.
    Choose an arbitrary node $r\in V(T)$ as an root.
    For each edge $t_1t_2\in E(T)$, where $t_1$ is the parent (with respect to $r$) of $t_2$, let $T_{t_1t_2}$ be the component of $T-t_1t_2$ containing $t_1$ (thus also containing $r$), and let $T_{t_2t_1}$ be other component.
    Let $G_{t_1t_2}=G-\bigcup_{t\in V(T_{t_1t_2})}\beta(t)$ and $G_2=G-\bigcup_{t\in V(T_{t_2t_1})}\beta(t)$.
    Notice that $G_{t_1t_2}$ and $G_{t_2t_1}$ are vertex-disjoint and no vertex of $G_{t_1t_2}$ can be adjacent to a vertex of $G_{t_2t_1}$ in $G$.

    Now, let $s_1s_2$ be an edge farthest from $r$ where $s_1$ is the parent of $s_2$, such that $G_{s_2s_1}$ contains $H$ as an induced minor.
    Notice that if such an edge does not exist, either $G$ doe snot contain $H$ as an induced minor, or $G-\beta(r)$ does not contain $H$ as an induced minor.
    Since $\alpha(\beta(r))\leq \tau(k(\ell+1),d)$ this would mean we are done.
    
    Hence, we may assume that such an edge $s_1s_2$ exists and observe that $G_{s_2s_1}-\beta(s_2)$ is the disjoint union of graphs $G_{s's_2}$ where $s'\in N(s_2)\setminus \{s_1\}$, so $G_{s_2s_1}-\beta(s-2)$ does not contain $H$ as an induced minor by choice of $s_1s_2$.
    
    If $G_{s_1s_2}$ contains $i-1$ disjoint copies of $H$ as an induced minor, then in total, $G$ contains $i$ disjoint copies of $H$ as an induced minor and we are done.
    Otherwise, by our induction hypothesis, $G_{s_1s_2}$ must contain a set $X_{i-1}$ such that $\alpha(X_{i-1})\leq (i-1)\tau(k(\ell+1,d))$ and $G_{s_1s_2}-X_{i-1}$ does not contain $H$ as an induced minor.
    Let $X_i=X_{i-1}\cup \beta(s_2)$.
    Then $\alpha(X_i)\leq \alpha(X_{i-1})+\alpha(\beta(s_2))\leq i\tau(k(\ell+1),d)$, and $G-X_i$ is the disjoint union of $G_{s_2s_1}-\beta(s_2)$ and $G_{s_1s_2}-X_{i-1}$.
    Hence, $G-X_i$ does not contain $H$ as an induced minor.
    This proves our claim, and by setting $i=k$ our theorem follows.
\end{proof}

\subsection{Excluding tripods and their line graphs}
Another type of substructure one can find in a $k$-rope ladder is the one of tripods and their line graphs.
A \emph{tripod} is a tree with at most three leaves.
Let $\mathcal{S}$ be the family of graphs whose components are tripods, and let $L(\mathcal{S})$ be the family of all line graphs of graphs from $\mathcal{S}$.

\begin{theorem}[\cite{dallard2024treewidthversuscliquenumber4}]\label{thm:TripodAlphaTW}
    For any positive integer $d$ and any two graphs $S\in \mathcal{S}$ and $T\in L(\mathcal{S})$, the class of $\{K_{1,d},S,T\}$-free graphs has bounded tree-independence number.
\end{theorem}

Before we dive into our proof for \cref{thm:TripodAlphaTW}, we define three types of adjacency between one endpoint of a path and the vertices of another path.
See \cref{fig:Junction types} for an illustration.
This definition will help us to identify tripods and their line graphs more easily.
\begin{figure}
    \centering
    \begin{tikzpicture}[scale=2.5]
    \pgfdeclarelayer{background}
    \pgfdeclarelayer{foreground}
    \pgfsetlayers{background,main,foreground}
    \begin{pgfonlayer}{background}
        \pgftext{\includegraphics[width=6cm]{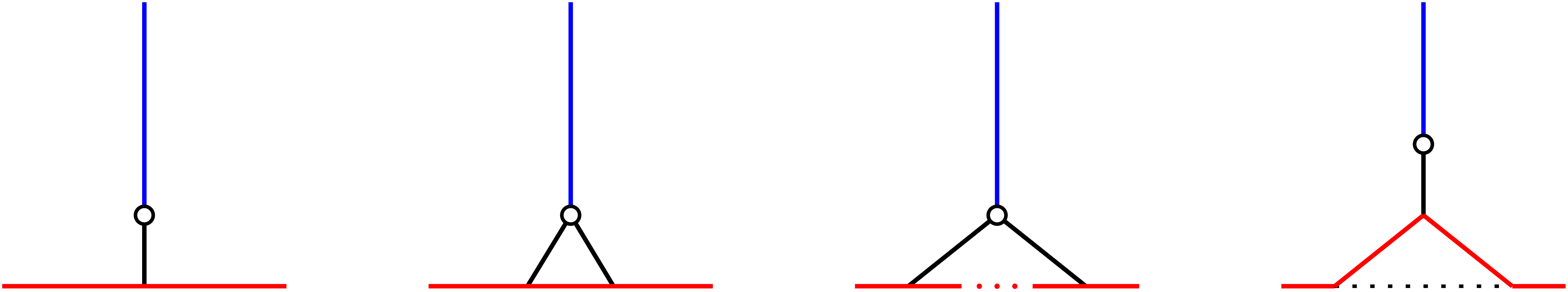}} at (C.center);
    \end{pgfonlayer}{background}
    \begin{pgfonlayer}{main}
        \node[]() at (-2.6,0.5) {$\color{blue}{P}$};
        \node[]() at (-0.98,0.5) {$\color{blue}{P}$};
        \node[]() at (0.64,0.5) {$\color{blue}{P}$};
        \node[]() at (2.1,0.5) {$\color{blue}{P-v}$};
        \node[]() at (-3.1,-0.48) {$\color{red}{Q}$};
        \node[]() at (-1.466,-0.48) {$\color{red}{Q}$};
        \node[]() at (0.158,-0.48) {$\color{red}{Q}$};
        \node[]() at (1.785,-0.48) {$\color{red}{Q'}$};
        \node[]() at (-2.6,-0.22) {$v$};
        \node[]() at (-0.97,-0.22) {$v$};
        \node[]() at (0.67,-0.22) {$v$};
        \node[]() at (2.32,-0.22) {$\color{red}{v}$};
        \node[]() at (2.16,-0.65) {$\color{red}{w_1}$};
        \node[]() at (2.82,-0.65) {$\color{red}{w_2}$};
    \end{pgfonlayer}{main}
    \begin{pgfonlayer}{foreground}
    \end{pgfonlayer}{foreground}
\end{tikzpicture}
    \caption{From the left: a junction connecting $P$ to $Q$ of \textbf{type 1}, \textbf{type 2}, \textbf{type 3}, and the cleaning of a junction of \textbf{type 3}.}
    \label{fig:Junction types}
\end{figure}
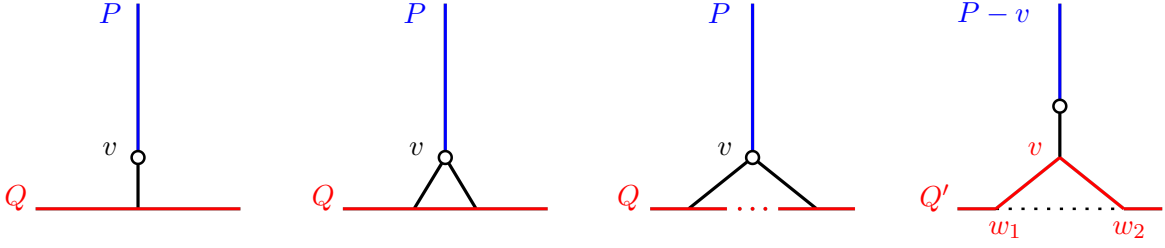

\begin{definition}
    Let $P,Q$ be disjoint induced paths such that $V(P)\cap N(Q)=\{v\}$ where $v$ is an endpoint of $Q$.
    We say that $N(P)\cap V(Q)$ is a \emph{junction connecting $P$ to $Q$}.
    Also, we say that:
    \begin{enumerate}
        \item If $|N(P)\cap V(Q)|=1$, we say that the junction is \emph{\textbf{type 1}}.
        \item If $|N(P)\cap V(Q)|=2$ and two vertices in $N(P)\cap V(Q)$ are adjacent, then we say the junction is \emph{\textbf{type 2}}.
        \item Otherwise, we say that the junction is \emph{\textbf{type 3}}.
    \end{enumerate}
\end{definition}

Suppose that the junction connecting $P$ to $Q$ is \textbf{type 3} where $P$ and $Q$ are induced paths and $V(P)\cap N(Q)=\{v\}$.
Let $w_1,w_2\in N(v)\cap P$ be the vertices such that $N(v)\cap V(Q)\subseteq V(w_1Qw_2)$.
Then $Q'\coloneqq Qw_1vw_2Q$ is an induced path since $w_1$ and $w_2$ are non-adjacent. 
If we also have $|V(P)|\geq 2$, then the junction connecting $P-v$ to $Q'$ is \textbf{type 1}.
We refer to this operation as \emph{cleaning}.

We now show how the existence of a $k$-rope ladder as an induced subgraph of a graph $G$ implies the presence of a large member of $\mathcal{S}\cup \mathcal{T}$ as an induced subgraph.

\begin{lemma}{\label{lem:tripodinladder}}
    For each $S\in \mathcal{S}$ and $T\in L(\mathcal{S})$, there exists an integer $k=k(S,T)$ such that every $k$-rope ladder graph contains one of $S$ or $T$ as an induced subgraph.
\end{lemma}
\begin{proof}
Let $S_p\in \mathcal{S}$ be the graph obtained from the $K_{1,3}$ by subdividing each edge $p-1$ times, and let $T_p\in \mathcal{T}$ be its line graph.
Note that $S_p$ contain a unique vertex with degree $3$.
Let us call this vertex as the \emph{central vertex}.
Also, $T_p$ contains three vertices with degree $3$, which forms a triangle,
Let us call this triangle the \emph{central triangle}.

Fix $S\in \mathcal{S}$ and $T\in L(\mathcal{S})$.
Then there exist integers $n$ and $p$ so that each component of $S$ is an induced subgraph of $S_p$ and each component of $T$ is an induced subgraph of $T_p$.
Thus, in the following it suffices to assume that $S=nS_p$ and $T=nT_p$.

Let $G$ be a $(2p+1)$-rope ladder, where $P_1,P_2$ are the two rail paths and $\Phi$ is the $(p+1)$-th rung path.
Let $\phi_1\in N(P_1)$ be an endpoint of $\Phi$.
If the junction connecting $\Phi^{p+1}$ to $P_1$ is \textbf{type 1}, then $G$ contains $S_p$ as an induced subgraph such that the central vertex is the unique vertex in $N(\phi_1)\cap V(P_1)$.
If the junction connecting $\Phi^{p+1}$ to $P_1$ is \textbf{type 2}, then $G$ contains $T_p$ as an induced subgraph such that the central triangle is formed by the two vertices in $N(\phi_1)\cap V(P_1)$ together with $\phi_1$.
Lastly, if the junction connecting $\Phi$ to $P_1$ is \textbf{type 3}, then $G$ contains $S_p$ as an induced subgraph such that the central vertex is $\phi_1$.
This last occurrence of $S_p$ is due to the cleaning operation described above.
Notice that in all three cases, we use $P_1$ to create two of the three subdivided paths of $S_p$ or $T_p$, and the remaining path will be found on $P_2$.

Finally, let $k \coloneqq (2p+2)(2n-1)$.
Then every $k$-ladder $H$ contains $(2n-1)$ pairwise non-adjacent $(2p+1)$-rope ladders as induced subgraphs.
Therefore, by our discussion above, $H$ contains either $n$ copies of $S_p$ or $T_p$ as an induced subgraph, which gives either $S$ or $T$ as an induced subgraph as desired.
\end{proof}

Notice that \cref{thm:TripodAlphaTW} is now directly implied by \cref{thm:twopaths}, \cref{thm:somethingpath}, and \cref{lem:tripodinladder}.

\subsection{Excluding an induced wheel minor}
This section is the first time where we consider specifically a cycle rope ladder instead of a rope ladder.
By \Cref{thm:pathandcycle}, every $K_{1,d}$-free graph $G$ of large enough $\alpha$-treewidth contains an induced cycle $C$ and an induced path $P$ such that their closed neighborhoods cannot be separated by a set of small independence number.
This allows us to choose $H=C$ in an application of \Cref{thm:somethingpath} which results in us obtaining a cycle rope ladder as an induced subgraph of $G$.

\begin{corollary}\label{thm:findingcycleropeladder}
    Let $k$ and $d\geq 2$ be positive integers.
    There exists a function $\tau\colon\mathbb{N}^2\to\mathbb{N}$ such that every $K_{1,d}$-free graph $G$ with $\atw(G)\geq \tau(k,d)$ contains a $k$-cycle rope ladder as an induced subgraph.
\end{corollary}

Let $W_\ell$ denote the wheel with $\ell$ spokes.
In \cite{choi2025excludinginducedwheelminor}, the authors showed that if a $K_{1,d}$-free graph has large -- in $\ell$ -- tree-independence number, it contains $W_{\ell}$ as an induced minor.

\begin{theorem}[\cite{choi2025excludinginducedwheelminor}]\label{thm:forbiddingthewheel}
    There exists a function $f(\ell,d)\in\mathcal{O}(d\ell^{11}+2^{(\max\{ \ell,d\})^5})$ such that for all positive integers $d$ and $\ell\geq 3$ and every $K_{1,d}$-free graph $G$, either
    \begin{enumerate}
        \item $G$ contains $W_{\ell}$ as an induced minor, or
        \item $\alpha\text{-}\mathsf{tw}(G)\leq f(\ell,d)$.
    \end{enumerate}
\end{theorem}

By using \Cref{thm:pathandcycle}, we can simplify the proof and improve the function $f$ in \Cref{thm:forbiddingthewheel} significantly.

\thmfindingwheel*
\begin{proof}
    Using \Cref{thm:pathandcycle}, we can find an induced path $P$ and an induced cycle $C$ such that $P$ and $C$ are non-adjacent and every separator between $N[P]$ and $N[C]$ has independence number at least $\ell (d-1)$ in $G$.
    Let $K$ be the component of $G-C$ containing $P$ and let $S \coloneqq N(C)\cap K$.
    Then $S$ is a separator between $N[C]$ and $N[P]$, so we have $\alpha(S)\geq \ell (d-1)$.
    Note that $N(K)=C\cap N(S)$.
    As $\ell (d-1) \leq \alpha(S)\leq (d-1)\cdot |N(K)|$, we have $|N(K)|\geq \ell$.
    Now select $N$ to be a set of $\ell$ neighbors of $K$ on $C$.
    Then, discard all components of $G-C$ other than $K$, contract $K$ into a single vertex, and contract edges of $C$ until the only remaining vertices of $C$ are precisely those of $N$.
    It is easy to see that the resulting graph is isomorphic to $W_{\ell}$.
\end{proof}

\subsection{Excluding long thetas and long prisms}
The last type of substructures we show to be unavoidable in $K_{1,d}$-free graphs of large tree-independence number are long thetas and long prisms.
In $\cite{chudnovsky2024treeindependencenumberiii}$, the authors proved that if we forbid thetas and generalized prisms as induced subgraphs in a $K_{1,d}$-free graph, then the tree-independence number is bounded.

\begin{theorem}[\cite{chudnovsky2024treeindependencenumberiii}]\label{thm:thetaandprism}
    For any positive integer $d$, the class of ($K_{1,d},\mathsf{Theta},\mathsf{Prism}$)-free graphs has bounded tree-independence number.
\end{theorem}

In what follows we present a strengthening of \cref{thm:thetaandprism} that replaces the class $\mathsf{Theta}$ by the class of $k$-long thetas and the class $\mathsf{Prism}$ by $k$-long prisms.

\thmklongthetaandprism*

By \Cref{thm:findingcycleropeladder}, it suffices to show that a large enough $K_{1,d}$-free cycle rope ladder either contains a $k$-long theta or a $k$-long prism as an induced subgraph.
As we need three paths between two vertices or two triangles to find a theta or a prism, our strategy is to find two of these paths as subpaths of the induced cycle of the cycle rope ladder, and to find the third one using the rail path.

\begin{lemma}\label{lem:cycleropeladderthetaandprism}
    Let $k$ and $d\geq 2$ be positive integers, and let $\ell\geq 12(d-1)^2(2k-1)+1$.
    Then every $\ell$-cycle rope ladder contains either a $k$-long theta or a $k$-long prism as an induced subgraph.
\end{lemma}

\begin{proof}
First, we fix our notation for the $\ell$-cycle rope ladder $G$ to be the same as in the definition:
Let $C$ be the rail cycle of $G$, $P$ be the rail path, and $\Phi^1,\Phi^2,\cdots,\Phi^\ell$ be the rung paths where $\phi_1^i\in N(C)$ and $\phi_2^i\in N(P)$ are the endpoints of $\Phi^i$ for $i\in [\ell]$.
Note that the length of $C$ is at least $\ell/(d-1)\geq 12(d-1)(2k-1)$.

We claim that there exists $A \subseteq [\ell]$ such that 
    \begin{itemize}
        \item $|A|= 4$,
        \item for each $i,j\in A$ with $i\neq j$, and for each $v\in N(\phi_1^i)\cap C$ and $w\in N(\phi_1^j)\cap C$, we have $\mbox{dist}_C(v,w)\geq k$, and
        \item for each $i,j\in A$ with $i\neq j$, and for each $v'\in N(\phi_2^i)\cap P$ and $w'\in N(\phi_2^j)\cap P$, we have $\mbox{dist}_P(v,w)\geq k$.
    \end{itemize}
To find such $A$, we iteratively choose an element $a\in [\ell]$ that has not yet been discarded and discard all elements of $[\ell]\setminus\{ a\}$ for which, if chosen next, one of the bottom two conditions would be violated.
All we need to do is to show that the number of elements removed each step by this way is bounded.
Say we chose $a\in [\ell]$ in one step.
Then we have $|N(\phi_1^a)\cap C|\leq 2(d-1)$, so the number of vertices in $C$ whose distance to $N(\phi_1^a)\cap C$ is less than $k$ is
$$\Big|\bigcup_{v\in N(\phi_1^a)\cap C}\{w\in V(C)\mid \mbox{dist}_C(v,w)<k\}\Big|\leq 2(d-1)(2k-1).$$
Since $\{\phi_1^i\}_{i\in [\ell]}$ is an independent set, each $w$ in the above set may adjacent to at most $d-1$ of the $\phi_1^i$'s.
Thus the number of $b\in [\ell]$ that cannot be in $A$ after choosing $a$ is at most
\begin{align*}
| \{b\in [\ell]\mid \text{there exists } v\in N(\phi_1^a)\cap C, \text{there exists } w\in N(\phi_1^b)\cap C \mbox{ such that } \mbox{dist}_C(v,w)<k \}|\\
\leq 2(d-1)^2(2k-1).
\end{align*}
We then apply the same procedure to $P$, so whenever we choose $a$, at most $4(d-1)^2(2k-1)$ other elements in $[\ell]$ become unavailable.
Since $\ell\geq 16(d-1)^2(2k-1)$, we are able to repeat this process at least four times, which proves the claim.
Without loss of generality, we may now assume $A=[4]$.
\medskip

Next we explain how to construct a $k$-long theta or a $k$-long prism according to the types of adjacency between $\phi_1^j$'s and $C$.

For each $i\in [4]$, we say that a path $D\subseteq C$ is a \emph{junction area of $\Phi^i$} if the endpoints of $D$ are in $N(\phi_1^i)\cap C$ for some $i$, $D$ does not contains any vertex in $N(\phi_1^j)\cap C$ for other $j\neq i$, and $D$ is maximal among such paths.
By the definition of $A$, the distances between two different junction areas in $C$ are at least $k$.
This guarantees that the paths we find from now on each have length at least $k$.

We proceed with a case distinction on the number of junction areas of $\Phi^i$ for each $i\in [4]$.

\begin{figure}
    \centering
    \begin{tikzpicture}[scale=2.5]
    \pgfdeclarelayer{background}
    \pgfdeclarelayer{foreground}
    \pgfsetlayers{background,main,foreground}
    \begin{pgfonlayer}{background}
        \pgftext{\includegraphics[width=6cm]{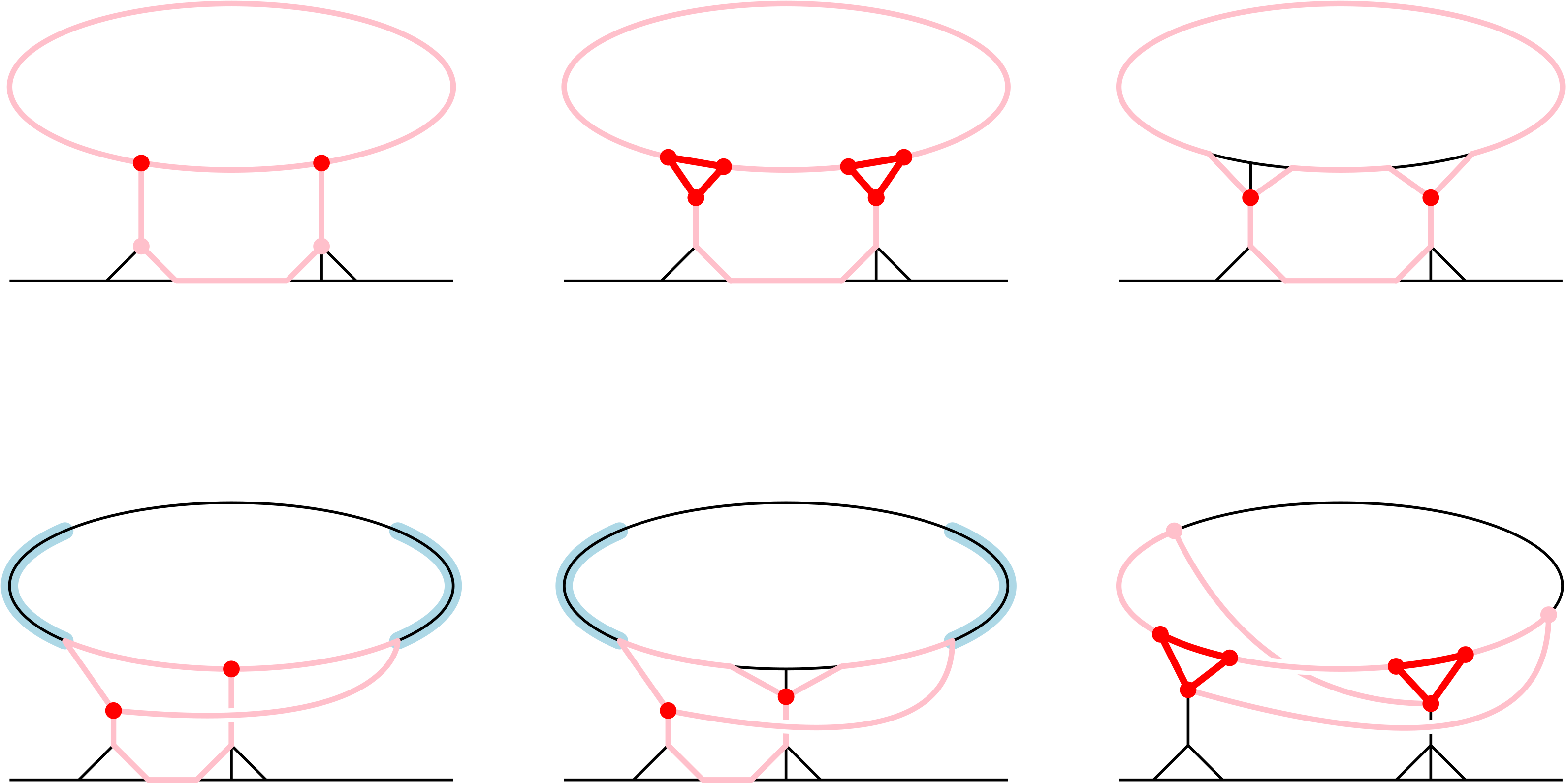}} at (C.center);
    \end{pgfonlayer}{background}
    \begin{pgfonlayer}{main}
        \node[]() at (-2.14,0.15) {\textbf{Case 1.1}} ;
        \node[]() at (-2.44,1) {\color{red}{$d_i$}} ;
        \node[]() at (-1.8,1) {\color{red}{$d_j$}} ;
        \node[]() at (-2.59,0.62) {\color{pink!90!black}{$\phi_2^i$}} ;
        \node[]() at (-1.65,0.62) {\color{pink!90!black}{$\phi_2^i$}} ;
        \node[]() at (-2.12,0.53) {\color{pink!90!black}{$\tilde{P}$}} ;
        \node[]() at (0,0.15) {\textbf{Case 1.2}} ;
        \node[]() at (-0.42,1) {\color{red}{$d_{i1}$}} ;
        \node[]() at (-0.16,0.96) {\color{red}{$d_{i2}$}} ;
        \node[]() at (0.17,0.96) {\color{red}{$d_{j1}$}} ;
        \node[]() at (0.43,1) {\color{red}{$d_{j2}$}} ;
        \node[]() at (-0.23,0.7) {\color{red}{$\phi_1^i$}} ;
        \node[]() at (0.46,0.71) {\color{red}{$\phi_1^j$}} ;
        \node[]() at (2.14,0.15) {\textbf{Case 1.3}} ;
        \node[]() at (1.9,0.7) {\color{red}{$\phi_1^i$}} ;
        \node[]() at (2.58,0.71) {\color{red}{$\phi_1^j$}} ;
        \node[]() at (-2.14,-1.8) {\textbf{Case 2.1}} ;
        \node[]() at (-2.68,-1.23) {\color{red}{$\phi_1^i$}} ;
        \node[]() at (-2.11,-0.93) {\color{red}{$d_j$}} ;
        \node[]() at (-2.9,-0.42) {\color{cyan}{$D_1^i$}} ;
        \node[]() at (-1.33,-0.42) {\color{cyan}{$D_2^i$}} ;
        \node[]() at (0,-1.8) {\textbf{Case 2.2}} ;
        \node[]() at (-0.555,-1.23) {\color{red}{$\phi_1^i$}} ;
        \node[]() at (0.13,-1.17) {\color{red}{$\phi_1^j$}} ;
        \node[]() at (-0.76,-0.42) {\color{cyan}{$D_1^i$}} ;
        \node[]() at (0.8,-0.42) {\color{cyan}{$D_2^i$}} ;
        \node[]() at (2.14,-1.8) {\textbf{Case 2.3}} ;
        \node[]() at (1.4,-0.4) {\color{pink!90!black}{$d_{i12}$}} ;
        \node[]() at (1.5,-0.81) {\color{red}{$d_{j11}$}} ;
        \node[]() at (1.8,-0.92) {\color{red}{$d_{j12}$}} ;
        \node[]() at (2.3,-0.94) {\color{red}{$d_{i21}$}} ;
        \node[]() at (2.6,-0.89) {\color{red}{$d_{i22}$}} ;
        \node[]() at (3.1,-0.9) {\color{pink!90!black}{$d_{j21}$}} ;
        \node[]() at (1.45,-1.2) {\color{red}{$\phi_1^j$}} ;
        \node[]() at (2.6,-1.2) {\color{red}{$\phi_1^i$}} ;
    \end{pgfonlayer}{main}
    \begin{pgfonlayer}{foreground}
    \end{pgfonlayer}{foreground}
\end{tikzpicture}
    \caption{The 6 cases towards finding a $k$-long theta or a $k$-long prism in a cycle rope ladder in the proof of \Cref{lem:cycleropeladderthetaandprism}.}
    \label{fig:thetaandprism}
\end{figure}
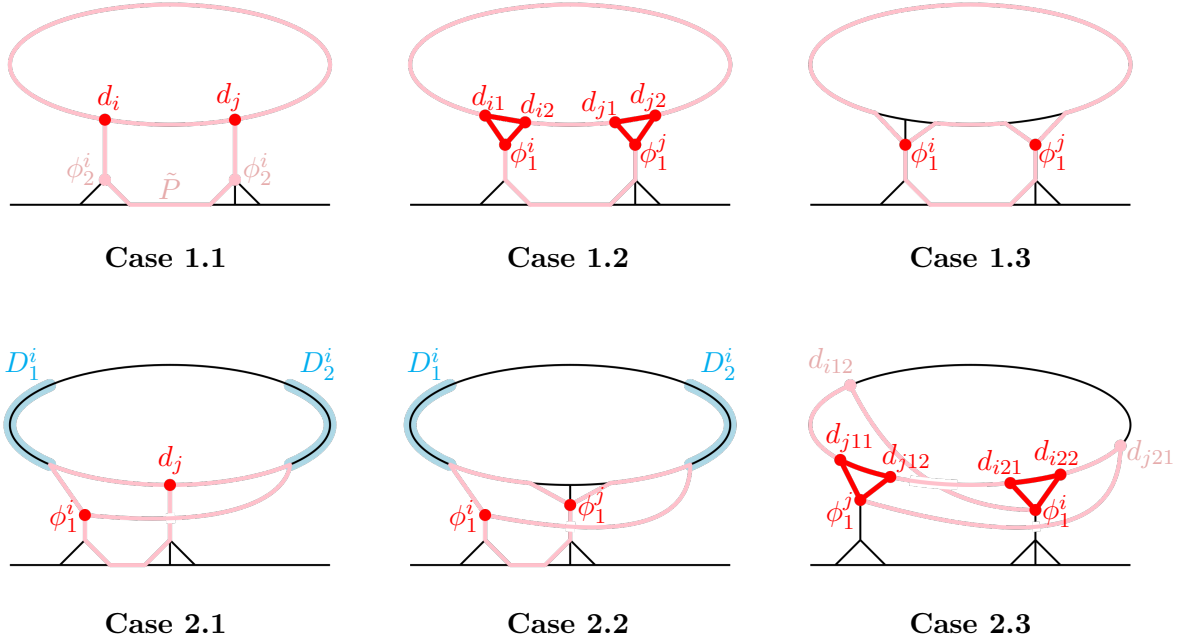

\paragraph{Case 1:} For each $i\in [4]$, there is a unique junction area of $\Phi^i$, say $D^i$.

\smallskip
Consider the type of junctions between $\Phi^i$ and $D^i$.
As there are three types of junctions and $|A|=4$, there must be $i\neq j$ such that $\Phi^i$ and $D^i$ has the same type of junction as $\Phi^j$ and $D^j$.
For any $i\in[A]$ we refer to the type of junction between $\Phi^i$ and $D^i$ as the \emph{type of $i$}

\paragraph{Case 1.1:} Both $i$ and $j$ are of \textbf{type 1}.
\smallskip

Let $V(D^i)=\{d_i\}$ and $V(D^j)=\{d_j\}$.
We now find a $k$-long theta as follows:
Let $d_i$ and $d_j$ be the ends of the theta.
Next, we choose two paths from $d_i$ to $d_j$ to be the two arcs of $C$ between them.
The third path of the theta is the path $d_i \phi_1^i \Phi^i \phi_2^i \widetilde{P} \phi_2^j \Phi^j \phi_1^j d_j$, where $\widetilde{P}$ is a shortest path between $\phi_2^i$ and $\phi_2^j$ in $G[P\cup \{\phi_2^i,\phi_2^j\}]$.
By the second part of the definition of $A$, this path has length at least $d$.
Except for the \textbf{Case 2.3}, one of the three paths will always be found in a similar way by routing through $P$.

\paragraph{Case 1.2:} Both $i$ and $j$ are of \textbf{type 2}.
\smallskip

Let $V(D^i)=\{d_{i1},d_{i2}\}$, $V(D^j)=\{d_{j1},d_{j2}\}$, where $d_{i1},d_{i2},d_{j1},d_{j2}$ appears according to the cyclic order on $C$.
Then a $k$-long prism can be constructed as follows:
The two triangles of the prism are $\{d_{i1},d_{i2},\phi_1^i\}$ and $\{d_{j1},d_{j2},\phi_1^j\}$.
We choose the path from $d_{i1}$ to $d_{j2}$ and the path from $d_{i2}$ to $d_{j1}$ to be the two internally disjoint arcs of $C$ between them.
The path from $\phi_1^i$ to $\phi_1^j$ can be obtained through $P$ as we did in \textbf{Case 1.1}.

\paragraph{Case 1.3:} Both $i$ and $j$ are of \textbf{type 3}.
\smallskip

Let $V(D^i)=\{d_{i1},\cdots,d_{ir}\}$ and let $V(D^j)=\{d_{j1},\cdots,d_{jr'}\}$, where $d_{i1},\cdots,d_{ir},d_{j1},\cdots,d_{jr'}$ appears in the cyclic order on $C$.
Then we can find a $k$-long theta as follows:
Let $\phi_1^i$ and $\phi_1^j$ be the ends of the theta.
We can choose the path from $d_{i1}$ to $d_{jr'}$ and the path from $d_{ir}$ to $d_{j1}$ to be the two disjoint arcs of $C$ between them.
By adding $\phi_1^i$ and $\phi_1^j$ as endpoints as depicted in \cref{fig:thetaandprism}, we can find two paths between the ends of the theta.
The third path can now be obtained through $P$ following the construction from \textbf{Case 1.1}.

\paragraph{Case 2:} There exists $i\in A$ such that $\Phi^i$ that has more than one junction area.
\smallskip

Let $D_1^i,\cdots,D_p^i$ be the junction areas of $\Phi^i$ appearing in the cyclic order on $C$.
Let $E_1^i,\cdots,E_p^i$ be the components of $C-\bigcup_{q\in [p]} D_q^i$ where $E_h^i$ is between $D_h^i$ and $D_{h+1}^i$ (considering $D_{p+1}^i=D_1^i$).
Then by the definition of junction areas, there is some $j\neq i\in A$ such that a junction area of $\Phi^j$, say $D_1^j$, is contained in one of the $E_q^i$'s, say $E_1^i$.

\paragraph{Case 2.1:} The junction connecting $\Phi^j$ to $E_1^i$ is of \textbf{type 1}.
\smallskip

In this case we have $|V(D_1^j)|=1$ and $D_1^j$ is the only junction area of $j$ contained in $E_1^i$.
Then we can find a $k$-long theta as following:
Let $\phi_1^i$ and $d_j\in V(D_1^j)$ be the ends of the theta.
The first and second path between the ends of the theta goes through two disjoint arcs of $C$ from $d_j$ to $D_1^i$ and $D_2^i$, respectively.
The third path can be obtained through $P$ by the similar way as before.

\paragraph{Case 2.2:} The junction connecting $\Phi^j$ to $E_1^i$ is of \textbf{type 3}.
\smallskip

Let $N(\Phi^j)\cap E_1^i=\{d_{j1},\cdots,d_{jr}\}$, where $D_1^i,d_{j1},\cdots,d_{jr},D_2^i$ appears in the cyclic order on $C$ as stated.
Note that $d_{j1}$ and $d_{jr}$ are not adjacent since the junction is of \textbf{type 3}.
We construct a $k$-long theta as follows:
Let $\phi_1^i$ and $\phi_1^j$ be the ends of the theta.
The first path can be obtained using an arc of $C$ from $d_{j1}$ to $D_1^i$, then adding $\phi_1^j$ as an endpoint.
The second path can be obtained using an arc of $C$ from $d_{jr}$ to $D_2^i$, and then adding $\phi_1^j$ as an endpoint.
Finally, the third path can be obtained through $P$ in the same way as we already did in \textbf{Case 1.1}.

\paragraph{Case 2.3:} The junction connecting $\Phi^j$ to $E_1^i$ is of \textbf{type 2}.
\smallskip

Furthermore we may assume that, since the two vertices in the junction are adjacent, only one junction area of $j$ is in $E_1^i$.
Similarly, we may assume for each $q\in [p]$, that each $E_q^i$ contains at most one junction area of $j\in [4]\setminus \{i\}$ and the junction connecting $\Phi^j$ to $E_q^i$ is of \textbf{type 2}.
This is because, if there is any possible choice for $j\in[4]\setminus\{ i\}$ such that there is a subpath $E$ of $C$ between two junction areas of $i$ and the junction of $\Phi^j$ to $E$ is \textsl{not} of \textbf{type 2}, we find ourselves in the situation of \textbf{Case 2.1} or \textbf{Case 2.2}.

If each $j\in [4]\setminus \{i\}$ has exactly one junction area, we may select two of them, remove the others, and then find ourselves in the same situation as \textbf{Case 1.2} which, as demonstrated above, implies the existence of a $k$-long prism.

So now we may assume that $j\in [4]\setminus \{i\}$ has two junction areas and is of \textbf{type 2}.

By switching the role of $i$ and $j$, we may further assume that $i$ has exactly one junction area between the junction areas of $j$, and their junctions are all of \textbf{type 2}.
In particular, the number of junction areas of $\Phi^i$ and $\Phi^j$ are the same, and the appearance of their junction areas alternates along $C$.
Let $D_1^j,\cdots,D_p^j$ be the junction areas of $j$, where $D_1^i,D_1^j,D_2^i,D_2^j,\cdots,D_p^i,D_p^j$ appear in the cyclic order on $C$ as listed.
Also let $D_1^i=\{d_{i11},d_{i12}\}$, $D_1^j=\{d_{j11},d_{j12}\}$, $D_2^i=\{d_{i21},d_{i22}\}$, and $D_2^j=\{d_{j21},d_{j22}\}$ where $d_{i12},d_{j11},d_{j12},d_{i21},d_{i22}$,and $d_{j21}$ appear in the cyclic order on $C$ as listed.

Now we may construct a $k$-long prism as follows:
The two triangles of our prism are $\{d_{j11},d_{j12},\phi_1^j\}$ and $\{d_{i21},d_{i22},\phi_1^i\}$.
The path from $d_{j11}$ to $\phi_1^i$ is obtained by taking the arc of $C$ from $d_{i12}$ to $d_{j11}$ not containing $d_{j12}$, and adding $\phi_1^i$ as an endpoint.
The path from $d_{j12}$ to $d_{i21}$ is obtained by taking the arc of $C$ between them not containing $d_{i22}$.
Similarly, the last path from $\phi_1^j$ to $d_{i22}$ is obtained by taking the arc of $C$ from $d_{i22}$ to $d_{j21}$ not containing $d_{i21}$, and adding $\phi_1^j$ as an endpoint.
See \cref{fig:thetaandprism} for an illustration.

Therefore, in each case, we found either a $k$-long theta or a $k$-long prism as an induced subgraph of $G$ as desired.
\end{proof}

\section{Conclusion}\label{sec:Conclusion}
In this section, we discuss a possiblity for a next step towards resolving \Cref{conj:K1dfreeinducedgrid}.
With \cref{thm:mainresult}, have constructed the first floor of the grid, so it seems natural to ask whether this result can be extended towards constructing the second floor.
For graphs $G$ and $H$, their \emph{Cartesian product} $G\openbox H$ is the graph with vertex set $V(G)\times V(H)$ where $(u_1,v_1)$ and $(u_2,v_2)$ are adjacent if and only if either $u_1=u_2$ and $v_1v_2\in E(H)$ or $v_1=v_2$ and $u_1u_2\in E(G)$.
A \emph{$k$-double ladder} is the graph obtained by the Cartesian product of the two paths $P_3\openbox P_k$.

\begin{figure}[h!]
    \centering
    \includegraphics[width=0.5\linewidth]{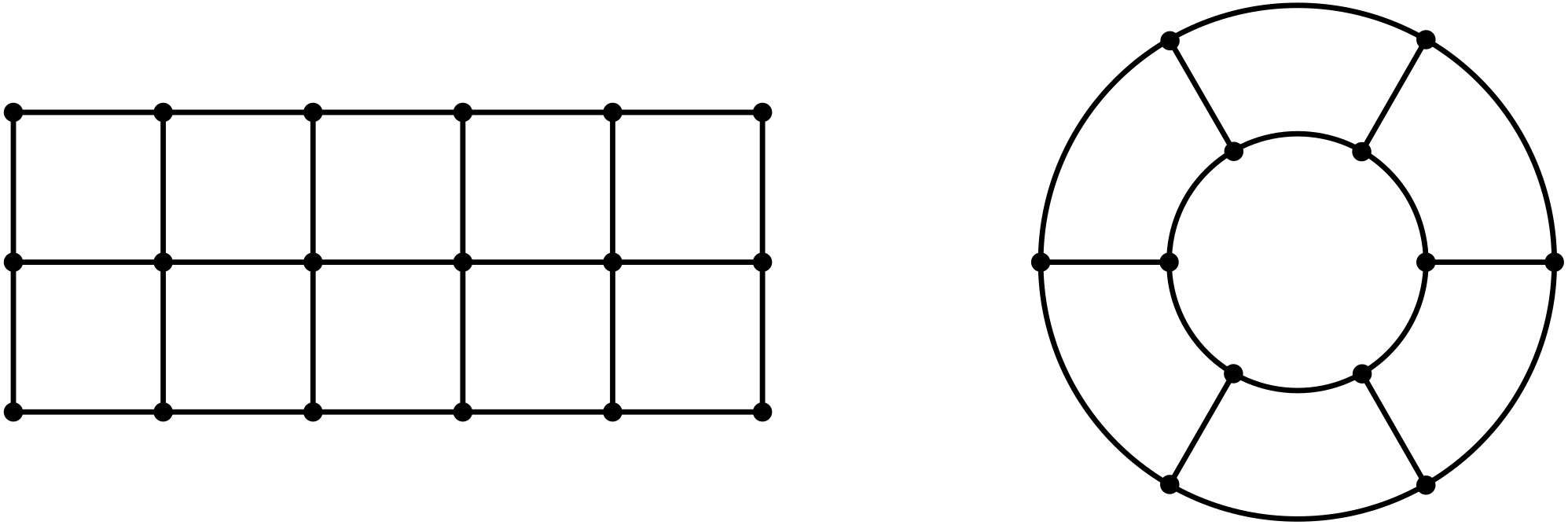}
    \caption{Possible next steps: a double ladder (left) and a double wheel (right)}
    \label{fig:doubleladderdoublewheel}
\end{figure}

\begin{conjecture}\label{conj:doubleladder}
    There exists a function $f_{\ref{conj:doubleladder}}\colon\mathbb{N}^2\rightarrow \mathbb{N}$ such that for every $K_{1,d}$-free graph $G$ with $\atw(G)\geq f_{\ref{conj:doubleladder}}(k,d)$ contains the $k$-double ladder as an induced minor.
\end{conjecture}

An immediate obstacle to using our method towards a possible resolution of \cref{conj:doubleladder} is that when we find a $k$-shuffled rope ladder, we cannot guarantee that each rung path has length at least $1$.
This seems relevant due to the following intuition:
We needed to find the $k$-skinny ladder as an intermediate step to find the $k$-ladder as an induced minor.
Similarly, it seems reasonable to suspect that one should find the graph obtained by subdividing each vertical edge of the $k$-double ladder once as an intermediate step towards finally finding the $k$-double ladder.
Such a graph contains, as an induced subgraph, a graph obtained by subdividing each rung path of the $k$-ladder at least twice; a ``\textsl{$k$-double skinny ladder}'' so to speak.
However, even finding this $k$-double skinny ladder seems to be a challenge.
Essentially, the difficulty boils down to the observation that we loose control over the independence number of the second neighbor of a fixed vertex set in $K_{1,d}$-free graphs.
Hence, replacing `non-adjacent' in our proof to `distance at least $2$' would require either require stronger assumptions than just $K_{1,d}$-freeness, or much more refined arguments for $K_{1,d}$-free graphs.

Another obstacle towards possible extensions of our result is the cleaning procedure applied in the proof of \Cref{lem:cleanshuffledropeladder} that allows us to go from a shuffled rope ladder to a non-shuffled one.
For example, in case of the cycle rope ladder, we cannot enforce the order in which the neighbors of the rung paths appear on the cycle to reflect the order in which they appear on the rail path. 
This is a stark contrast to \Cref{lem:cleanshuffledropeladder} and mostly due to the fact that we need to be able to sort our overlapping ``junction areas'' between the rung paths and the cycle.
In \Cref{thm:somethingpath}, if we choose $H$ to be a more complex graph, for example a ladder, it will be much harder to find an ordered substructure between $H$ and its neighbors.
\medskip

Another possible next step toward a grid is the double wheel.
A \emph{$k$-double wheel} is the graph $C\openbox P_2$, where $C$ is a cycle on $k$ vertices.
A double wheel could be considered as the first floor of a cylindrical grid and would ``only'' require to take the $k$-ladder and somehow ``close'' it to a cycle.

\begin{conjecture}\label{conj:doublewheel}
    There exists a function $f_{\ref{conj:doublewheel}}:\mathbb{N}^2\rightarrow \mathbb{N}$ such that for every $K_{1,d}$-free graph $G$ with $\atw(G)\geq f_{\ref{conj:doublewheel}}(k,d)$ contains the $k$-double wheel as an induced minor.
\end{conjecture}

In an ideal world, this could be achieved by adapting our methods to finding many pairwise non-adjacent paths between two long cycles instead of two long induced paths.
It seems reasonable to expect an analogue of \cref{thm:pathandcycle} that finds two induced cycles that are non-adjacent but cannot be separated by a set of small independence number.
However, it seems that one would need to adjust our arguments in a non-trivial way to achieve such a result as discussed below.

\begin{conjecture}\label{conj:twocycles}
    There exists a function $f_{\ref{conj:twocycles}}:\mathbb{N}^2\rightarrow \mathbb{N}$ such that for every $K_{1,d}$-free graph $G$ with $\atw(G)\geq f_{\ref{conj:twocycles}}(\eta,d)$ contains two induced cycles $C_1,C_2$ such that
    \begin{itemize}
        \item $C_1$ and $C_2$ are non-adjacent, and
        \item if $S$ is a separator between $C_1$ and $C_2$, then $\alpha(S)\geq \eta$.
    \end{itemize}
\end{conjecture}

Unfortunately, we cannot directly follow the proof of \Cref{thm:pathandcycle}.
In the proof of \Cref{thm:pathandcycle}, after using two bramble elements to form the first cycle, a second cycle cannot be formed by the same argument.
This is because the path connecting the endpoints of the second path might be adjacent to the first cycle.

Finally, changing the path or the cycle in \Cref{thm:pathandcycle} to some other useful graphs would be another interesting problem.

\paragraph{Acknowledgments.}
We are grateful to Maximilian Gorsky for helpful discussions and great company.

\bibliographystyle{alphaurl}
\bibliography{Reference}

\end{document}